\documentclass[12 pt]{article}
\usepackage{stmaryrd}
\usepackage{times}
\usepackage{booktabs}
\usepackage{subfigure}
\usepackage{longtable}
\usepackage{setspace}
\usepackage{tabularx}
\usepackage{multirow, makecell}
\usepackage{rotating}
\usepackage{pifont}
\usepackage{floatrow}
\usepackage{caption}
\usepackage{mathrsfs}
\usepackage[fleqn]{amsmath}
\usepackage{amsfonts,amsthm,amssymb,mathrsfs,bbding}
\usepackage{txfonts}
\usepackage{graphics,multicol}
\usepackage{graphicx}
\usepackage{color}
\usepackage{caption}
\usepackage{indentfirst}
\usepackage{cite}
\usepackage{latexsym,bm}
\usepackage{enumerate}
\usepackage{epsfig}
\evensidemargin=\oddsidemargin\topmargin=-1.5cm

\newtheorem{thm}{Theorem}[section]
\newtheorem{lem}{Lemma}[section]
\newtheorem{cor}{Corollary}[section]
\newtheorem{pro}{Proposition}[thm]

\theoremstyle{definition}
\newtheorem{defi}{Definition}[section]
\newtheorem{remark}{Remark}[section]

\addtocounter{section}{0}
\begin{document}
\title{The extremal graphs of order trees and  their topological indices\footnote{This work is supported
by NSFC Grant No. 11971274, 11531011, 11671344 .}}
\author{
{\small Rui Song$^{a}$,\ \ Qiongxiang Huang$^{a,}$\footnote{Corresponding author.
\newline{\it \hspace*{5mm}Email addresses:} huangqx@xju.edu.cn(Q.X. Huang).}\ \   and Peng Wang$^{a}$}\\[2mm]
\footnotesize $^a$College of Mathematics and Systems Science, Xinjiang University, Urumqi, Xinjiang 830046, China\\
}
\date{}
\maketitle {\flushleft\large\bf Abstract:}
Recently, D. Vuki$\check{c}$evi$\acute{c}$ and  J. Sedlar in \cite{Vuki} introduced an order ``$\preceq$" on $\mathcal{T}_n$, the set of  trees on $n$ vertices, such that the topological index $F$ of a graph is a function defined on the order set $\langle\mathcal{T}_n,\preceq\rangle$. It provides a new approach to determine the extremal graphs with respect to  topological index $F$. By using the method they determined the common  maximum and/or minimum graphs of $\mathcal{T}_n$ with respect to topological indices  of   Wiener type and anti-Wiener type. Motivated by their researches we further study the  order set $\langle\mathcal{T}_n,\preceq\rangle$ and give a criterion to determine its order, which enable us to get the common extremal graphs in four  prescribed subclasses of $\langle\mathcal{T}_n,\preceq\rangle$. All these extremal graphs  are  confirmed to be the common maximum and/or minimum graphs with respect to the  topological indices  of   Wiener type and anti-Wiener type. Additionally, we calculate the exact values of Wiener index for the extremal graphs in the order sets $\langle\mathcal{C}(n,k),\preceq\rangle$, $\langle\mathcal{T}_{n}(q),\preceq\rangle$ and $\langle\mathcal{T}_{n}^\Delta,\preceq\rangle$.
\begin{flushleft}

\end{flushleft}
\section{Introduction}
The Wiener index of a graph $G$ defined by $W(G) =\sum_{u,v\in V}d(u, v)$ is the  first topological index introduced early in 1947 by H.Wiener in \cite{Wiener}. With time other topological indices, such as  some variations and modifications of Wiener index, are introduced and studied because of their chemical applications and mathematical properties \cite{Gutman-1,Nikoli,Liu1,Chen,Vukic}. Some of the topological indices are related with distance $d(u,v)$ as in $W(G)$. The researches for such  topological indices are usually focus on some classes of graphs, especially trees, and produce many results published in various academic journals, one can refer to \cite{Bonchev,Liu,Wang,Pandey,Nikoli,Liu1,Chen,Zhang2,Vukic} for references. All theses results   mainly include  evaluating the bounds of the topological indices and characterizing  the corresponding extremal graphs. Recently, D. Vuki$\check{c}$evi$\acute{c}$ and J. Sedlar in \cite{Vuki} introduced an  order on trees by edge division vector and  established the relationship between  the order and  topological indices. It  allows us to study this problem in a uniform way and then  obtain general conclusions.

Let $G$ be a simple graph on $n$ vertices. For a vertex $v\in V(G)$, let $N_G(v)$ be the neighbors of $v$ and $d_G(v)=|N_G(v)|$ the degree of $v$. A vertex $v$ is called a pendent vertex if $d_G(v)=1$.   For a pair  of vertices $u,v\in V(G)$ we define distance $d_G(u,v)$ as the length of the shortest path connecting vertices $u$ and $v$.

A tree $T$ is a connected graph without cycle. A vertex  $v$ of $T$ is a branching vertex if $d_T(v)\geq 3$. Given a tree $T$, let $e=uv\in E(T)$ be an edge, $T_u$ and $T_v$ be respectively the two components of $T-e$ containing $u$ and $v$. By $n_u(e)$ (resp., $n_v(e)$) we  denote the number of vertices whose distance to vertex $u$ (resp., $v$) is smaller than the distance to vertex $v$ (resp., $u$), i.e., $n_u(e)=|T_u|$ and  $n_v(e)=|T_v|$. Also, we  write  $n_u'(e)$ and $n_v'(e)$ for the tree $T'$. Furthermore, for an edge $e=uv\in E(T)$ we define $\mu(e)=\min\{n_u(e),n_v(e)\}$. By definition, we have $n_u(e)+n_v(e)=n$. It follows that $\mu(e)\leq \lfloor\frac{n}{2}\rfloor$.

Let $\mathcal{T}_n$ denote the set of  trees on $n$ vertices. For a tree $T\in \mathcal{T}_n$, let $r_i(T)$ denote the number of edges  for which $\mu(e)=i$, i.e., $r_i(T)=|\{e\in E(T)\mid \mu(e)=i\}|$. It is clear that $r_1(T)$ is just the number of pendent edges of $T$ and  $r_i(T)=0$ for every $i>\lfloor\frac{n}{2}\rfloor$ due to  $\mu(e)\leq \lfloor\frac{n}{2}\rfloor$. The edge division vector $\mathbf{r}(T)$ of a tree $T$ is defined as a vector
$\mathbf{r}(T)=(r_1(T),r_2(T),\ldots,r_{\lfloor\frac{n}{2}\rfloor}(T))$. We will write only $\mathbf{r}$ and $r_i$ when it does not lead to confusion.
Let $\mathbf{r}$ and $\mathbf{r}'$ be  respectively the edge division vectors of $T$ and $T'$ in $ \mathcal{T}_n$. Recently, D. Vuki$\check{c}$evi$\acute{c}$ and J. Sedlar in \cite{Vuki}  defined an order: $\mathbf{r}\preceq \mathbf{r}'$ if and only if  the inequality $\sum\limits_{i=k}^{\lfloor\frac{n}{2}\rfloor}r_i\leq \sum\limits_{i=k}^{\lfloor\frac{n}{2}\rfloor}r'_i$
holds for every $k=1,2,\ldots,\lfloor\frac{n}{2}\rfloor$. If the inequality is strict for at least one $k$, then we say that $\mathbf{r}\prec \mathbf{r}'$. Naturally, they  introduced an  order of trees by using edge division vector as following: for $T,T'\in\mathcal{T}_n$,   $T\preceq T'$ if $\mathbf{r}\preceq \mathbf{r}'$ (with $T\prec T'$ if $\mathbf{r}\prec \mathbf{r}'$). It is worth mentioning that such an order   ``$\preceq$" satisfies the reflexivity and  transitivity but  not antisymmetry, and so is not  partial order (one can refer to the Remark \ref{rem-1}), i.e.,  $\langle\mathcal{T}_n,\preceq\rangle$ is not a poset but  order set. Thus by $T\approx T'$ we mean $\mathbf{r}(T)=\mathbf{r}(T')$. However, if a tree $T$ is uniquely determined by its edge division vector $\mathbf{r}(T)$, then $\mathbf{r}(T)=\mathbf{r}(T')$ implies that $T=T'$.

Since $\mathcal{T}_n$ is  finite set, for any subset $\mathcal{S}\subseteq\mathcal{T}_n$,  $\langle\mathcal{S},\preceq\rangle$  must have a unique maximum element $T^+$ called  the maximal graph (with respect to $\mathcal{S}$), and  minimum element $T^-$ called  the minimal graph  (with respect to $\mathcal{S}$). Also in \cite{Vuki} D. Vuki$\check{c}$evi$\acute{c}$ and J. Sedlar established the relationship between $\langle\mathcal{T}_n,\preceq\rangle$ and topological index $F$ of trees by
defining  $F:\mathcal{T}_n\rightarrow \mathbb{R}$ with the  monotonous properties that for $T,T'\in \mathcal{T}_n$, if  $T\preceq T'$ implies $F(T)\leq F(T')$ (resp., $F(T)\geq F(T')$)  then $F$ is  refereed to a topological index of  Wiener type (resp., anti-Wiener type). Thus if $F$ is confirmed to be a topological index of  Wiener type, then $T^+$  will be the maximum  graph for this index, and $T^-$ the minimum graph while the reverse holds for a topological index of anti-Wiener type. This gives us a general approach to find the corresponding maximum  or minimum graph and evaluate bounds of topological index $F$ in some class. By using the  approach, D. Vuki$\check{c}$evi$\acute{c}$ and J. Sedlar in \cite{Vuki} proved that the $P_n$ and $S_n$ are respectively the maximal  and minimal elements in $\langle\mathcal{T}_n,\preceq\rangle$. Moreover, by verifying the monotonicity of the topological index, they confirm that $P_n$ and $S_n$ are also the common maximum graph and/or minimum graph with respect to the indices of  Wiener index, Steiner $k$-Wiener index,  modified Wiener indices and  variable Wiener indices, respectively.

Inspired  by their research works  in \cite{Vuki}, in this paper we first give a criterion to determine the  orders of trees, which leads to some  graph transformations  preserving   this  order in $\langle\mathcal{T}_n,\preceq\rangle$, and then we determine the  maximum graph and/or  minimum graph in  the following classes:\\
(1) $\mathcal{C}(n,k)$, the set of caterpillar trees on $n$ vertices with respect to the path $P_k$.\\
(2) $\mathcal{T}_n(q)$, the subset of $\mathcal{T}_n$ with exactly $q$ pendent vertices.\\
(3) $\mathcal{T}(n,k-1)$, the subset of $\mathcal{T}_n$ with diameter $k-1$.\\
(4) $\mathcal{T}_n^\Delta$, the subset of $\mathcal{T}_n$ with maximum degree $\Delta$.\\
Next, by verifying the monotonicity of the topological indices, we show that the above extremal graphs are respectively the common maximum graph and/or  minimum graph  with respect to the following topological indices: \\
(1) The hyper-Wiener index proposed by M. Randi$\acute{c}$ in \cite{Rand}, i.e.,
$$WW(T)=\sum_{\{u,v\}\in V(T)}\dbinom{1+d(u,v)}{2}=
\sum_{\{u,v\}\in V(T)}[\frac{1}{2}d(u,v)+\frac{1}{2}d^2(u,v)],$$
where $d(u,v)$ is the distance between  vertices $u$ and $v$.\\
(2) The Wiener-Hosoya index proposed by M. Randi$\acute{c}$ in \cite{Randi}, i.e.,
$$h(T)=\sum_{e\in E(T)}(h(e)+h[e]),$$
where $h(e)$ is the product of the numbers of the vertices in each component of $T-e$, and $h[e]$ is the product of the numbers of the vertices in each component of $T-\{u,v\}$ ($\{u,v\}$ are two end vertices of $e$).\\
(3) The degree distance proposed by D. J. Klein et al. in \cite{Klein}, i.e.,
$$D'(T)=\sum_{\{u,v\}\in V(T)}(d(u)+d(v))d(u,v)=4\sum_{\{u,v\}\in V(T)}d(u,v)-n(n-1).$$
\\
(4) The Gutman index proposed by I. Gutman in \cite{Gutman}, i.e.,
$$Gut(T)=\sum_{\{u,v\}\subseteq V(T)}d(u)d(v)d(u,v)=4\sum_{\{u,v\}\in V(T)}d(u,v)-(2n-1)(n-1).$$\\
(5) The second atom-bond connectivity index proposed by A. Graovac and M. Ghorbani in \cite{Graovac}, i.e.,
$$ABC_2(T)=\sum_{e=uv\in E(T)}\sqrt{\frac{n_u(e)+n_v(e)-2}{n_u(e)n_v(e)}},$$
where $n_u(e)$ (resp., $n_v(e)$) denote the number of vertices whose distance to vertex $u$ (resp., $v$) is smaller than the distance to vertex $v$ (resp., $u$).

The present paper is organized as follows. In section 2, some notions  and properties related with tree are introduced. In section 3, we  give a lemma that  is our criterion  to determine the  order of trees. By using this criterion we can set up  some  graph transformations that  preserve   the  order in $\langle\mathcal{T}_n,\preceq\rangle$. In section 4, we give the maximum and minimum graphs in  $\langle\mathcal{C}(n,k),\preceq\rangle$.  In section 5, we give the  minimum graph in  $\langle\mathcal{T}(n,k-1),\preceq\rangle$. In section 6, we give the maximum and minimum graphs in  $\langle\mathcal{T}_{n}(q),\preceq\rangle$. In section 7, we give the maximum  graph in $\langle\mathcal{T}_{n}^\Delta,\preceq\rangle$.  Finally, in section 8, we prove that all the extremal graphs mentioned above are the common  maximum graph and/or  minimum graph with respect to  all the topological indices of  $WW(T)$,  $h(T)$, $D'(T)$, $Gut(T)$ and $ABC_2(T)$. Finally, we give a Table 3 that summarize all the related results old and new. At last of this section, as an example to determine the bounds of topological indices,  we calculate the exact values of Wiener index for the extremal graphs in the order sets $\langle\mathcal{C}(n,k),\preceq\rangle$, $\langle\mathcal{T}_{n}(q),\preceq\rangle$ and $\langle\mathcal{T}_{n}^\Delta,\preceq\rangle$.

\section{Preliminaries}

Let $T$ be a tree of order $n$. A vertex $u$ of $T$ is said to be \emph{centroidal vertex} of $T$  if $n_u(e)\geq \frac{n}{2}$ for any edge $e\in E(T)$ incident to $u$, and a centroidal vertex $u$ is said to be proper if $n_u(e)> \lceil\frac{n}{2}\rceil$ for any edge $e\in E(T)$ incident to $u$. For a centroidal vertex $u$, let $N_T(u)=\{v_1,v_2,\ldots,v_r\}$ and denote by $T_{v_i}(uv_i)$  the component of $T-uv_i$ containing $v_i$. It is clear that $u$ is a centroidal (resp., proper centroidal)  vertex of $T$  if and only if $|T_{v_i}(uv_i)|\le \frac{n}{2}$ (resp., $|T_{v_i}(uv_i)|< \lfloor\frac{n}{2}\rfloor$) for $i=1,2,\ldots,r$.

For a tree $T$ of order $n=2,3$, it is clear that  $T$ has   centroidal vertices. For $n>3$, let $v$ be a pendent vertex of $T$ adjacent to $v'$, and then $T'=T-v$ has centroidal vertex  $u$  by  induction hypothesis. Let $N_{T'}(u)=\{v_1,v_2,\ldots,v_r\}$ and $T_{v_i}'(uv_i)$ be the component of $T'-uv_i$ containing $v_i$. By assumption, for each $1\le i\le r$ we have
$$|T_{v_i}'(uv_i)|=n_{v_i}'(uv_i)=n-1-n_u'(uv_i)\le n-1-\frac{n-1}{2}=\frac{n-1}{2}< \frac{n}{2}.$$
On the other aspect, without loss of generality, assume that  $v'$ belongs to  $T_{v_1}'(uv_1)$. Then $T_{v_1}(uv_1)=T_{v_1}'(uv_1)+v'v$ and $T_{v_i}(uv_i)=T_{v_i}'(uv_i)$ are also the components of $T-u$ containing $v_i$ for $i=2,3,\ldots,r$. If $|T_{v_1}'(uv_1)|<\lfloor\frac{n}{2}\rfloor$, then $|T_{v_i}(uv_i)|\le \lfloor\frac{n}{2}\rfloor$ for all $i=1,2,\ldots,r$, and thus $u$ is also a centroidal vertex of $T$. Otherwise, $|T_{v_1}'(uv_1)|=\lfloor\frac{n}{2}\rfloor$. We have $|T_{v_1}(uv_1)|=\lfloor\frac{n}{2}\rfloor+1$ and  thus $|T_{v_2}(uv_2)|+\cdots+|T_{v_r}(uv_r)|=n-1-|T_{v_1}(uv_1)|=n-2-\lfloor\frac{n}{2}\rfloor<\lfloor\frac{n}{2}\rfloor$. It implies that  by replacing $v$ with a pendent vertex of $T$ in $T_{v_2}(uv_2)$,  the centroidal vertex $u$ of $T-v$ is also a centroidal vertex of $T$.  Therefore, the centroidal vertex of $T$ always exists.

Suppose that  $u_1$ and $u_2$ be two  centroidal vertices of $T$ and $P=u_1\cdots u_2$ be the path connecting $u_1$ and $u_2$. We claim that $P=u_1u_2$. Since otherwise, let $P=u_1x_1\cdots x_2u_2$ where $x_1$ may be equal to $x_2$. We have
$n\le\lceil\frac{n}{2}\rceil+\lceil\frac{n}{2}\rceil\le |T_{u_1}(u_1x_1)|+|T_{u_2}(u_2x_2)|\leq n-1 $, a contradiction. It also implies that $T$ has at most two centroidal vertices. Summering above arguments, we get the following simple results mentioned in \cite{Ma,Jordan}.

\begin{lem}\label{lem-2}
Let $T$ be a tree. Then the following statements hold:\\
(1) $T$ has at least one and at most two centroidal vertices;\\
(2) $T$ has two centroidal vertices $u$ and $v$ if and only if  $uv$ is an edge of $T$ such that the two components of $T-uv$ have the same order.
\end{lem}

For a tree $T$ on $n$ vertices, we say an edge $e=uv\in E(T)$ is \emph{center edge} if $\mu(e)=\min\{n_{u}(e),n_{v}(e)\}=\lfloor\frac{n}{2}\rfloor$.  Obviously,   $S_n$, the star on $n$ vertices, has no any center edge if $n\ge4$. For a center edge $e=uv$, without loss of generality, we assume that $n_{u}(e)\le n_{v}(e)$. Then $\lfloor\frac{n}{2}\rfloor =n_{u}(e)\le n_{v}(e)=\lceil\frac{n}{2}\rceil$ since $n_{u}(e)+n_{v}(e)=n$. To exactly, we have
\begin{equation}\label{center-eq-1}\left\{\begin{array}{ll}
n_{u}(e)=\frac{n}{2} = n_{v}(e),& \mbox{ if $n$ is even }\\
\lfloor\frac{n}{2}\rfloor =n_{u}(e)< n_{v}(e)=\frac{n+1}{2}=\lceil\frac{n}{2}\rceil,& \mbox{ if $n$ is odd. }\\
\end{array}\right.
\end{equation}
It implies the following simple result.
\begin{lem}\label{lem-3-0}
Let $T$ be a tree on $n$ vertices. Then the following statements hold:\\
(1) If $T$ has  center edge $e=uv$ and $n$ is even, then $n_{u}(e)= n_{v}(e)=\frac{n}{2}$ and  $uv$ is unique.\\
(2) If $T$ has  center edge $e=uv$ and $n$ is odd, then $\lfloor\frac{n}{2}\rfloor =n_{u}(e)< n_{v}(e)=\lceil\frac{n}{2}\rceil$. \\
(3) $T$ has two center edges $e_1$ and $e_2$ if and only if  they  are adjacent with common vertex $z=e_1\cap e_2$  such that  $ n_z(e_1)=n_z(e_2)=\lceil\frac{n}{2}\rceil
$.\\
(4) $T$ has at most two center edges.
\end{lem}
\begin{proof}
(1) and (2) follows immediately from Eq. (\ref{center-eq-1}).

First we prove (3). On the contrary suppose that $e_1=u_1u_2$ and $e_2=v_{1}v_2$ are  center edges not adjacent. From (1) and (2), we know that $n$ is odd and may assume that
$$\left\{\begin{array}{ll}
\lfloor\frac{n}{2}\rfloor =n_{u_1}(e_1)< n_{u_2}(e_1)=\lceil\frac{n}{2}\rceil,\\
\lfloor\frac{n}{2}\rfloor =n_{v_1}(e_2)< n_{v_2}(e_2)=\lceil\frac{n}{2}\rceil.\\
\end{array}\right.
$$
It implies that $u_2=v_2$. Thus $e_1$ and $e_2$ are adjacent with common vertex $u_2=v_2=z$, consequently $ n_z(e_1)=n_z(e_2)=\lceil\frac{n}{2}\rceil$. Conversely, suppose that  $e_1$ and $e_2$ have common vertex $z=e_1\cap e_2$  such that  $ n_z(e_1)=n_z(e_2)=\lceil\frac{n}{2}\rceil$. Let $e_1=uz$. We have $ n_u(e_1)=\lfloor\frac{n}{2}\rfloor$ and so $\mu_T(e_1)=\min\{n_u(e_1),n_z(e_1)\}=\lfloor\frac{n}{2}\rfloor$. Similarly, $e_2$ is also  a center edge with $\mu_T(e_2)=\lfloor\frac{n}{2}\rfloor$.

One can simply verify (4) from (3).
\end{proof}

The following lemma gives some relations between centroidal vertex and center edge.

\begin{lem}\label{lem-4}
Let $T$ be a tree on $n$ vertices. Then the following statements hold:\\
(1) If $T$ has two centroidal vertices $u$ and $v$, then $uv$ is the only center edge of $T$;\\
(2) If $T$ has  only one proper centroidal vertex if and only if $T$ has no center edge;\\
(3) If $T$ has only one centroidal vertex $u$ such that $n_{u}(e)=\lceil\frac{n}{2}\rceil$ for any edge $e$ incident to $u$ if and only if $d(u)=2$ and $T$ has two center edges incident to $u$;\\
(4) If $T$ has only one centroidal vertex $u$ such that $n_{u}(e)=\lceil\frac{n}{2}\rceil$ for only one edge $e$ incident to $u$, then $e$ is just  one center edge of $T$.
\end{lem}
\begin{proof}
Suppose that $u$ and $v$ are two centroidal vertices of $T$. Then  $uv$ is an edge of $T$ and the two components of $T-e$ have the same order by Lemma \ref{lem-2} (2). It follows that $uv$ is a center edge and $n$ is even.  By Lemma \ref{lem-3-0} (1), the center edge $uv$ is unique and (1) follows.

In the following proofs of (2), (3) and (4), we always assume that the centroidal vertex $u$ has neighbors $v_1,\ldots,v_r$ and $T_i$ is the component of $T-u$ containing $v_i$ for $i=1,\ldots,r$.

Now we show (2).  Suppose that $T$ has  one proper centroidal vertex $u$. Then $n_{u}(uv_i)>\lceil\frac{n}{2}\rceil$, and so $uv_i$ is not center edge since $\mu(uv_i)=\min\{n_{u}(uv_i),n_{v_i}(uv_i)\}<\lfloor\frac{n}{2}\rfloor$. For the edge $e\in E(T_i)$, we have $\mu(e)<|T_i|=n_{v_i}(uv_i)<\lfloor\frac{n}{2}\rfloor$. Thus, $T$ has no center edge.
To show the converse, we assume that $\mu(e)<\lfloor\frac{n}{2}\rfloor$ for all the edge $e\in E(T)$.  For $v\in V(T)$, $N_T(v)$ can be partitioned as two parts $N_T(v)=N^-(v)\cup N^+(v)$ such that $N^-(v)=\{x\in N_T(v)\mid n_x(xv)< \lfloor\frac{n}{2}\rfloor\}$ and $N^+(v)=\{y\in N_T(v)\mid n_y(yv)\ge \lfloor\frac{n}{2}\rfloor\}$. Now let $v^*$ be a vertex such that $|N^+(v^*)|$ is as small as possible. In what follows we show that $u=v^*$ is a proper centroidal vertex. It is clear that $v^*$ is a centroidal vertex if  $|N^+(v^*)|=0$. On the contrary, assume that $|N^+(v^*)|\ge 1$. We claim that $|N^+(v^*)|=1$. Since otherwise, there are two vertices $y_1,y_2\in N^+(v^*)$. We have $n_{y_1}(y_1v^*), n_{y_2}(y_2v^*)\ge \lfloor\frac{n}{2}\rfloor$, and furthermore $n_{y_1}(y_1v^*), n_{y_2}(y_2v^*)>\lceil\frac{n}{2}\rceil$ since there is no any center edge. Thus $n\ge n_{y_1}(y_1v^*)+n_{y_2}(y_2v^*)+1> \lceil\frac{n}{2}\rceil+\lceil\frac{n}{2}\rceil+1\geq n+1$, a contradiction. Now we can further assume that $y_1$ is the unique vertex in $ N^+(v^*)$ and $v^*$ is the unique vertex in $N^+(y_1)$ since $T$ is finite and acyclic. Thus $n_{y_1}(v^*y_1)\ge \lfloor\frac{n}{2}\rfloor$ and $n_{v^*}(y_1 v^*)\ge \lfloor\frac{n}{2}\rfloor$. It implies that $\{n_{y_1}(v^*y_1),n_{v^*}(y_1 v^*)\}=\{\lfloor\frac{n}{2}\rfloor,\lceil\frac{n}{2}\rceil\}$, and so $v^*y_1$ is a center edge of $T$. It is a contradiction.

Next we show (3). Suppose that $T$ has only one centroidal vertex $u$ such that $n_{u}(e)=\lceil\frac{n}{2}\rceil$ for any edge $e$ incident to $u$. We have   $|T_i|=\lfloor\frac{n}{2}\rfloor$ for $i=1,2,\ldots,r$. It  implies that  $r=2$ and $d(u)=2$. Therefore, $\mu(uv_i)=|T_i|=\lfloor\frac{n}{2}\rfloor$, and  $uv_1$ and $uv_2$ are the two center edges.
To show the converse, let $e_1=uv_1$ and $e_2=uv_2$ be two center edges.  We have $n_u(e_1)=n_u(e_2)=\lceil\frac{n}{2}\rceil$ by Lemma \ref{lem-3-0} (3). Thus the vertex $u$ is the unique  centroidal vertex.

At last we show (4). Without loss of generality, we say $n_{u}(uv_1)=\lceil\frac{n}{2}\rceil$ and $n_{u}(uv_i)>\lceil\frac{n}{2}\rceil$ for $2\leq i\leq r$. We have $|T_1|=\lfloor\frac{n}{2}\rfloor$ and $|T_i|<\lfloor\frac{n}{2}\rfloor$ for $2\leq i\leq r$. Since  $\mu(uv_1)=|T_1|=\lfloor\frac{n}{2}\rfloor$, $uv_1$ is a center edge. In addition,   $\mu(uv_i)=|T_i|<\lfloor\frac{n}{2}\rfloor$ for $i=2,\ldots,r$, and  $\mu(e)<|T_i|\leq\lfloor\frac{n}{2}\rfloor$ for any edge $e\in E(T_i)$ where $i=1,\ldots,r$. Therefore, any edge other than $e=uv_1$ is not center edge of $T$.

We complete this proof.
\end{proof}

\section{ A criterion to determine the  order of trees }

In this section, we will give a criterion to determine the  order  in $\langle\mathcal{T}_n, \preceq\rangle$.

For $T,T'\in \mathcal{T}_n$, let $\varphi:E(T)\longrightarrow E(T')$ be a bijection. $T$ and $T'$ are said to be $(\varphi,\mu)$-similar with respect to $e_1\in E(T)$ if   $\mu_{T}(e)=\mu_{T'}(\varphi(e))$ for any $e\not=e_1$.

\begin{lem}\label{lem-5}
Suppose that  $T, T'\in \mathcal{T}_n$  are $(\varphi,\mu)$-similar with respect to $e_1$, and $\varphi(e_1)=e_1'$. We have\\
(1) If $\mu_{T}(e_1)<\mu_{T'}(e_1')$, then $T\prec T'$; \\
(2) If $\mu_{T}(e_1)>\mu_{T'}(e_1')$, then $T\succ T'$;\\
(3) If $\mu_{T}(e_1)=\mu_{T'}(e_1')$, then $T\approx T'$.
\end{lem}
\begin{proof}
Let $\mathbf{r}=(r_1,\ldots,r_{\lfloor\frac{n}{2}\rfloor})$ and $\mathbf{r}'=(r_1',\ldots,r_{\lfloor\frac{n}{2}\rfloor}')$ be the edge division vectors of $T$ and $T'$, respectively.  By assumption, there exists a bijection $\varphi:E(T)\longrightarrow E(T')$ such that $\mu_{T}(e)=\mu_{T'}(\varphi(e))$ for any $e\not=e_1$.

Now we prove (1). Without loss of generality, assume that $s=\mu_{T }(e_1)<\mu_{T'}(e_1')=t$. First, for  $i\not=s,t$ we have
$$\begin{array}{ll}r_i=|\{e\in E(T)\mid\mu_{T}(e)=i\}|&=|\{e\in E(T\backslash e_1)\mid\mu_{T}(e)=i\}|\\
&=|\{\varphi(e)\in E(T'\backslash e_1')\mid\mu_{T'}(\varphi(e))=i\}|\\
&=|\{e'\in E(T')\mid\mu_{T'}(e')=i\}|=r_i'.
\end{array}
$$
Next, for $i=s$ we have
\[\begin{split}
r_s&=|\{e\in E(T)\mid\mu_{T}(e)=s\}|\\
&=|\{e\in E(T)\backslash e_1\mid\mu_{T}(e)=s\}|+1\\
&=|\{\varphi(e)\in E(T')\backslash e_1'\mid\mu_{T'}(\varphi(e))=s\}|+1\\
&=|\{e'\in E(T')\mid\mu_{T'}(e')=s\}|+1\\
&=r_s'+1.
\end{split}\]
At last, for $i=t$ we can similarly get $r_t'=r_t+1$ by exchanging the positions of $s$ and $t$. It follows that
$$\left\{\begin{array}{ll}
\sum_{i=l}^{\lfloor\frac{n}{2}\rfloor}r_{i}=\sum_{i=l}^{\lfloor\frac{n}{2}\rfloor}r_{i}'& \mbox{ for $1 \leq l \leq s$, }\\
\sum_{i=l}^{\lfloor\frac{n}{2}\rfloor}r_{i}<\sum_{i=l}^{\lfloor\frac{n}{2}\rfloor}r_{i}' & \mbox{ for $s+1 \leq l \leq t$, }\\
\sum_{i=l}^{\lfloor\frac{n}{2}\rfloor}r_{i}=\sum_{i=l}^{\lfloor\frac{n}{2}\rfloor}r_{i}'&
\mbox{ for $t+1 \leq l \leq \lfloor\frac{n}{2}\rfloor$.}
\end{array}\right.
$$
Therefore, we have $\mathbf{r}\prec \mathbf{r}'$, and so  $T\prec T'$ by definition. Thus (1) follows.

To prove (2), without loss of generality, assume that $t=\mu_{T }(e_1)>\mu_{T'}(e_1')=s$.   By   exchanging the positions of $T$ and $T'$, and  similarly as the proof of  (1) one can verify that $r_i= r'_i$ for $i\not=s,t$,  $r'_s=r_s+1$ and $r_t= r'_t+1$, which leads to $T\succ T'$, and so (2) follows.

According to definitions, (3) is obvious.

We complete this proof.
\end{proof}
\begin{remark}\label{rem-1}
Obviously, if $\mathbf{r}(T)\not=\mathbf{r}(T')$ then $T\not=T'$. Conversely, when $\mathbf{r}(T)=\mathbf{r}(T')$, however the following example shows that $T$ does not necessarily equal $T'$. It implies that  the  order ``$\preceq$" defined on $\mathcal{T}_n$ by the edge division vector is not a partial order since the antisymmetry of this order does not hold. This gives us a method to characterize graphs they are not isomorphic, but with the same index values (we can refer to Section 8).
\end{remark}

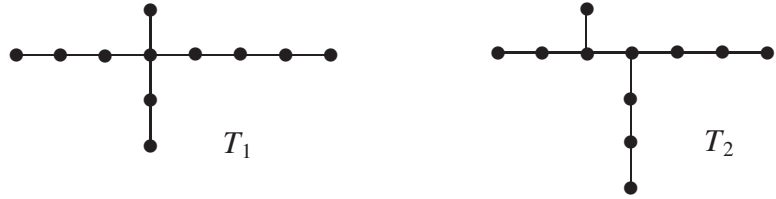
\begin{figure}[H]
\unitlength 1.5mm 
\linethickness{0.4pt}
\ifx\plotpoint\undefined\newsavebox{\plotpoint}\fi 
\begin{picture}(67.978,17.464)(0,0)
\put(.56,12.614){\line(1,0){27.96}}
\put(.665,12.614){\circle*{1.2}}
\put(4.554,12.614){\circle*{1.2}}
\put(8.548,12.509){\circle*{1.2}}
\put(12.542,12.614){\circle*{1.2}}
\put(16.537,12.719){\circle*{1.2}}
\put(20.531,12.719){\circle*{1.2}}
\put(24.525,12.614){\circle*{1.2}}
\put(28.519,12.614){\circle*{1.2}}
\put(12.542,16.608){\line(0,-1){11.878}}
\put(12.542,16.608){\circle*{1.2}}
\put(12.542,8.62){\circle*{1.2}}
\put(12.542,4.52){\circle*{1.2}}
\put(43.375,12.786){\circle*{1.2}}
\put(47.264,12.786){\circle*{1.2}}
\put(51.258,12.678){\circle*{1.2}}
\put(55.252,12.786){\circle*{1.2}}
\put(59.247,12.888){\circle*{1.2}}
\put(63.241,12.888){\circle*{1.2}}
\put(67.235,12.786){\circle*{1.2}}
\put(51.119,16.713){\line(0,-1){3.889}}
\put(51.231,16.721){\circle*{1.2}}
\put(55.139,12.897){\line(0,-1){11.878}}
\put(55.139,4.907){\circle*{1.2}}
\put(55.139,.8){\circle*{1.2}}
\put(55.113,8.725){\circle*{1.2}}
\put(43.13,12.824){\line(1,0){24.281}}
\put(18.954,3.89){$T_1$}
\put(61.63,4.1){$T_2$}
\end{picture}
  \caption{The graphs $T_1$ and $T_2$ with the same edge division vector $\mathbf{r}=(4,3,2,1,0)$ but $T_1\neq T_2$}\label{fig-1}
\end{figure}

For  $T\in \mathcal{T}_n$, let  $P_k=v_1\cdots v_k$ be a path of $T$ and $U(v_i)$ be the neighbors of $v_i$ not belong to $P_k$, where $1\le i\le k$. For $x_i\in U(v_i)$, let $T_{x_i}(v_ix_i)$ be the component of $T-v_ix_i$ containing $x_i$, where  $x_i$ is defined to be the root vertex of the branch $T_{x_i}(v_ix_i)$,  and $T_{v_i}=\cup_{x_i\in U(v_i)}T_{x_i}(v_ix_i)$ is a union of   all such  branches.  Let  $n_i=|T_{v_i}|=\sum_{x_i\in U(v_i)}|T_{x_i}(v_ix_i)|$.  Obviously, $n=k+\sum_{1\le i\le k}n_i$.  Such a tree $T$ is denoted by $T(n;n_1,\ldots,n_{k})$ and called the tree with respect to path  $P_{k}=v_1\cdots v_k$. All such graphs are collected in $\mathcal{T}_{n;n_1,\ldots,n_{k}}$.

Let $T=T(n;n_1,\ldots,n_{k})\in \mathcal{T}_{n;n_1,\ldots,n_{k}}$, where $n_t> 0$ and $1\le t<k$. We call $T'=T(n;n_1,\ldots,n_{t-1},0,n_{t+1}+n_t,\ldots,n_k)$ (see Figure \ref{fig-2}) be the branch-shift  of $T$ from $v_t$ to $v_{t+1}$, that can be viewed as a graph transformation from $T$ by shifting branch $T_{v_t}$ to $v_{t+1}$ (equivalently,  deleting edges $v_tx_i$ and adding  $v_{t+1}x_i$ for any $x_i\in U(v_t)$).

\begin{figure}[H]
\unitlength 0.8mm 
\linethickness{0.4pt}
\ifx\plotpoint\undefined\newsavebox{\plotpoint}\fi 
\begin{picture}(160.004,43.96)(0,0)
\put(20.545,25){\line(1,0){30}}
\put(5.544,25){\circle*{2}}
\put(20.545,25){\circle*{2}}
\put(35.544,25){\circle*{2}}
\put(50.544,25){\circle*{2}}
\put(65.544,25){\circle*{2}}
\put(10.544,25){\ldots}
\put(55.544,25){\ldots}
\put(8.002,0){\small$T=T(n;n_1,\ldots,n_t,\ldots,n_k)$}
\put(3.545,28){$v_{1}$}
\put(18.545,28){$v_{t-1}$}
\put(33.544,28){$v_{t}$}
\put(48.544,28){$v_{t+1}$}
\put(63.544,28){$v_{k}$}
\put(86.001,0){\small$T'=T(n;n_1,\ldots,n_{t-1},0,n_{t+1}+n_t,\ldots,n_k)$}
\put(109.668,24.625){\line(1,0){30}}
\put(94.668,24.625){\circle*{2}}
\put(109.668,24.625){\circle*{2}}
\put(124.669,24.625){\circle*{2}}
\put(139.668,24.625){\circle*{2}}
\put(154.668,24.625){\circle*{2}}
\put(99.668,24.625){\ldots}
\put(144.668,24.625){\ldots}
\put(92.668,27.625){$v_{1}$}
\put(107.668,27.625){$v_{t-1}$}
\put(122.669,27.625){$v_{t}$}
\put(139,27.625){$v_{t+1}$}
\put(152.668,27.625){$v_{k}$}
\put(3,8){\footnotesize$T_{v_1}$}
\put(16,8){\footnotesize$T_{v_{t-1}}$}
\put(33,8){\footnotesize$T_{v_t}$}
\put(46,8){\footnotesize$T_{v_{t+1}}$}
\put(63,8){\footnotesize$T_{v_k}$}
\put(137,39){\footnotesize$T_{v_t}$}
\put(74.794,24.375){\vector(1,0){7.625}}
\put(5.481,11.181){\circle{10.961}}
\put(20.482,11.168){\circle{10.961}}
\put(35.481,11.23){\circle{10.961}}
\put(50.544,11.355){\circle{10.961}}
\put(65.545,11.23){\circle{10.961}}
\put(92.001,8){\footnotesize$T_{v_1}$}
\put(105.001,8){\footnotesize$T_{v_{t-1}}$}
\put(135,8){\footnotesize$T_{v_{t+1}}$}
\put(152,8){\footnotesize$T_{v_k}$}
\put(94.46,10.876){\circle{10.961}}
\put(109.46,10.862){\circle{10.961}}
\put(154.523,10.926){\circle{10.961}}
\put(139.483,11.23){\circle{10.961}}
\put(139.483,38.48){\circle{10.961}}
\put(2.8,11.3){\small$\cdots$}
\put(17.7,11.3){\small$\cdots$}
\put(32.5,11.3){\small$\cdots$}
\put(47.8,11.3){\small$\cdots$}
\put(62.7,11.1){\small$\cdots$}
\put(92.001,11.1){\small$\cdots$}
\put(107.001,11.1){\small$\cdots$}
\put(137,11.1){\small$\cdots$}
\put(152,11.3){\small$\cdots$}
\put(136.8,36.2){\small$\cdots$}
\multiput(5.501,25.125)(-.033482143,-.112723214){112}{\line(0,-1){.112723214}}
\multiput(5.501,24.875)(.033415842,-.126237624){101}{\line(0,-1){.126237624}}
\put(8.701,12.5){\circle*{1}}
\put(1.75,12.5){\circle*{1}}
\multiput(20.501,25.25)(-.033482143,-.112723214){112}{\line(0,-1){.112723214}}
\multiput(20.501,25)(.033415842,-.126237624){101}{\line(0,-1){.126237624}}
\put(23.8,12.5){\circle*{1}}
\put(16.9,12.5){\circle*{1}}
\multiput(35.376,25.375)(-.033482143,-.112723214){112}{\line(0,-1){.112723214}}
\multiput(35.376,25.125)(.033415842,-.126237624){101}{\line(0,-1){.126237624}}
\put(38.576,12.5){\circle*{1}}
\put(31.626,12.5){\circle*{1}}
\multiput(50.501,25.375)(-.033482143,-.112723214){112}{\line(0,-1){.112723214}}
\multiput(50.501,25.125)(.033415842,-.126237624){101}{\line(0,-1){.126237624}}
\put(53.701,12.5){\circle*{1}}
\put(46.751,12.5){\circle*{1}}
\multiput(65.376,25.5)(-.033482143,-.112723214){112}{\line(0,-1){.112723214}}
\multiput(65.376,25.25)(.033415842,-.126237624){101}{\line(0,-1){.126237624}}
\put(68.576,12.5){\circle*{1}}
\put(61.626,12.5){\circle*{1}}
\multiput(94.626,24.875)(-.033482143,-.112723214){112}{\line(0,-1){.112723214}}
\multiput(94.626,24.625)(.033415842,-.126237624){101}{\line(0,-1){.126237624}}
\put(97.826,12.3){\circle*{1}}
\put(90.876,12.3){\circle*{1}}
\multiput(109.751,24.75)(-.033482143,-.112723214){112}{\line(0,-1){.112723214}}
\multiput(109.751,24.5)(.033415842,-.126237624){101}{\line(0,-1){.126237624}}
\put(112.951,12.3){\circle*{1}}
\put(106.001,12.3){\circle*{1}}
\multiput(139.751,24.75)(-.033482143,-.112723214){112}{\line(0,-1){.112723214}}
\multiput(139.751,24.5)(.033415842,-.126237624){101}{\line(0,-1){.126237624}}
\put(142.951,12.3){\circle*{1}}
\put(136.001,12.3){\circle*{1}}
\multiput(154.626,24.875)(-.033482143,-.112723214){112}{\line(0,-1){.112723214}}
\multiput(154.626,24.625)(.033415842,-.126237624){101}{\line(0,-1){.126237624}}
\put(157.826,12.3){\circle*{1}}
\put(150.876,12.3){\circle*{1}}
\multiput(139.626,24.375)(-.033536585,.109756098){123}{\line(0,1){.109756098}}
\multiput(139.626,24.625)(.033505155,.136597938){97}{\line(0,1){.136597938}}
\put(135.626,37.5){\circle*{1}}
\put(143,37.5){\circle*{1}}
\end{picture}
  \caption{The branch-shift transformation}\label{fig-2}
\end{figure}
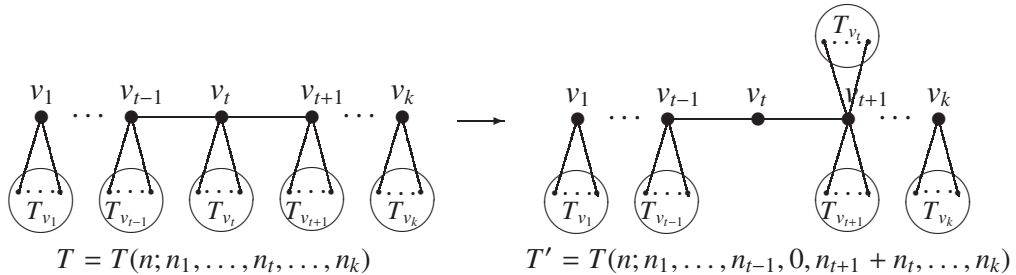

In the following  sections, we will characterize the extremal graphs with respect to the order ``$\preceq$" on some  prescribed families of trees.
For $T=T(n;n_1,\ldots,n_{k})\in \mathcal{T}_{n;n_1,\ldots,n_{k}}$ with respect to  $P_k=v_1\cdots v_k$, we know that  $T_{v_i}$ is a union of the branches $T_{x_i}(v_ix_i)$ for $x_i\in U(v_i)$. If each $T_{v_i}$ is an independent set then $T$ is so called the caterpillar tree which is written as  $CP(n;n_1,\ldots,n_{k})$ instead of $T(n;n_1,\ldots,n_{k})$. It is clear that  $\mathcal{T}_{n;n_1,\ldots,n_{k}}$ contains  a unique caterpillar tree $CP(n;n_1,\ldots,n_{k})$. Given $n>k$, all the caterpillar trees of the form $CP(n;n_1,\ldots,n_{k})$ satisfying  $n=n_1+n_2+\cdots+n_k+k$ are collected in $\mathcal{C}(n,k)$.  In the next section, we will determine the extremal  graphs in  $\langle\mathcal{C}(n,k),\preceq\rangle$.

\section{Extremal  graphs in the order set $\langle\mathcal{C}(n,k),\preceq\rangle$}

The \emph{caterpillar tree} $T=CP(n;n_1,\ldots,n_{k})$  with respect to a path $P_{k}=v_1v_2\cdots v_k$ is a tree obtained from $P_{k}$ by attaching $n_i$ pendent vertices at $v_i$ for $i=1,2,\ldots,k$, where $n=\sum_{i=1}^{k}n_i+k$. In this section, we always assume that $1\le n_1\leq n_k $ by symmetry. Thus, if $k=n$ then $T=P_n$; if $k=1$ then $T=S_n$. Since $P_n$ and $S_n$ are unique, we may assume that $2\le k\le n-1$ in what follows.

The caterpillar tree $T=CP(n;n_1,\underbrace{0,\ldots,0}_{k-2},n_k)$  is called  the \emph{double star path} and shortly for $CP_{n;n_1,n_k}$. We call
$T'=CP_{n;n_1+1,n_k-1}$ be the edge-shift of $T=CP_{n;n_1,n_k}$, that can be viewed as a graph transformation from $T$ by shifting  edge  $v_{k}x$ to $v_{1}x$ (see Figure \ref{fig-3}). First we give a lemma that will be used to determine the extremal  graphs in $\langle\mathcal{C}(n,k),\preceq\rangle$.

\begin{figure}[H]
\unitlength 1mm 
\linethickness{0.4pt}
\ifx\plotpoint\undefined\newsavebox{\plotpoint}\fi 
\begin{picture}(134.462,17.256)(0,0)
\put(15.057,12.25){\circle*{2}}
\put(25.057,12.25){\circle*{2}}
\put(4.157,10.25){$\vdots$}
\put(24.057,8.75){$v_{2}$}
\put(21.432,0){$T=CP_{n;n_1,n_k}$}
\put(36.115,12.318){\circle*{2}}
\put(46.115,12.318){\circle*{2}}
\put(54.415,10){$\vdots$}
\put(35.115,8.818){$v_{k-1}$}
\put(45.115,8.818){$v_{k}$}
\put(15.095,12.136){\line(1,0){9.811}}
\put(35.906,12.136){\line(1,0){10.406}}
\put(28.257,10.987){$\cdots$}
\put(91.815,12.169){\circle*{2}}
\put(101.815,12.169){\circle*{2}}
\put(80.315,10.169){$\vdots$}
\put(100.815,8.669){$v_{2}$}
\put(112.873,12.237){\circle*{2}}
\put(122.873,12.237){\circle*{2}}
\put(131.473,10.237){$\vdots$}
\put(111.873,8.737){$v_{k-1}$}
\put(121.873,8.737){$v_{k}$}
\put(91.852,12.055){\line(1,0){9.811}}
\put(112.664,12.055){\line(1,0){10.406}}
\put(104.815,11.206){$\cdots$}
\put(97.258,.69){$T'=CP_{n;n_1+1,n_k-1}$}
\put(-3,10.636){\small$n_1$}
\put(0,10.436){$\Bigg \{$}
\multiput(15.095,12.041)(-.103472678,.033519318){102}{\line(-1,0){.103472678}}
\multiput(15.244,12.041)(-.071270974,-.033599174){146}{\line(-1,0){.071270974}}
\multiput(46.163,12.041)(.082739646,.033656805){106}{\line(1,0){.082739646}}
\multiput(45.717,12.338)(.064309959,-.033736372){141}{\line(1,0){.064309959}}
\multiput(91.947,12.338)(-.11508456,.03356633){93}{\line(-1,0){.11508456}}
\multiput(91.799,12.189)(-.069234661,-.033599174){146}{\line(-1,0){.069234661}}
\multiput(122.867,12.189)(.073126647,.03356633){124}{\line(1,0){.073126647}}
\multiput(123.015,12.041)(.057478344,-.033694202){150}{\line(1,0){.057478344}}
\put(4.689,15.608){\circle*{2}}
\put(4.689,6.987){\circle*{2}}
\put(54.933,15.162){\circle*{2}}
\put(54.785,7.284){\circle*{2}}
\put(80.799,15.608){\circle*{2}}
\put(131.934,16.352){\circle*{2}}
\put(131.637,7.284){\circle*{2}}
\put(80.799,7.284){\circle*{2}}
\put(91.15,8.7){$v_{1}$}
\put(14.649,8.7){$v_{1}$}
\put(56.798,9.811){\Bigg \}}
\put(59.798,9.96){\small$n_k$}
\put(63,11){\vector(1,0){4}}
\put(67.7,9.96){\small$n_1+1$}
\put(77.253,10.2){\Bigg \{}
\put(135.5,9.9){\small$n_k-1$}
\put(132.962,10.2){\Bigg\}}
\end{picture}
 \caption{The edge-shift transformation}\label{fig-3}
\end{figure}
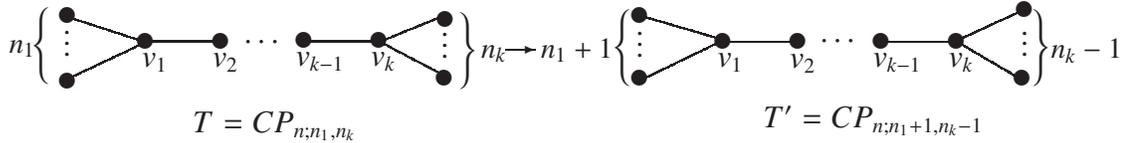

\begin{lem}\label{lem-7}
If $n_1+2\leq n_k$, then  $T=CP_{n;n_1,n_k}\prec T'=CP_{n;n_1+1,n_k-1}$.
\end{lem}
\begin{proof}
Since $T'$ is the edge-shift of $T$, we may assume that the pendent edge $v_kx_1$ of $T$ is shifted  to $v_1 x_1$.
We define bijection $\varphi: E(T)\longrightarrow E(T')$ such that
\begin{equation}\nonumber \label{simi-eq-2}\varphi(e)=\left\{\begin{array}{ll}
e, & \mbox{ if $e\neq v_kx_1$ is a pendent edge of $T$}\\
v_1x_1, & \mbox{ if $e= v_kx_1$ }\\
v_{k-1}v_k, & \mbox{ if $e=v_1v_2$ }\\
v_{j-1}v_{j}, &  \mbox{ if $e=v_jv_{j+1}$ where $2\leq j\leq k-1$.}
\end{array}\right.
\end{equation}
It is clear that $\mu_{T}(e)=\mu_{T'}(\varphi(e))=1$ if $e$ is a pendent edge of $T$. In what follows, we will prove our result by distinguishing three situations to verify the condition of Lemma \ref{lem-5}.

{\flushleft\bf Case 1.} $T$ has no center edge.

By  Lemma \ref{lem-4} (2), we see  that $v_k$ must be the proper centroidal vertex in $T$. Thus,
$$\left\{\begin{array}{ll}
\mu_{T}(v_jv_{j+1})=min\{n_1+j,n_k+k-j\}=n_1+j<\lfloor\frac{n}{2}\rfloor,\\
\mu_{T'}(v_jv_{j+1})=min\{n_1+j+1,n_k+k-j-1\}=n_1+j+1,
\end{array}\right.
$$
where $1\leq j\leq k-1$. It follows that  $\mu_{T}(v_jv_{j+1})=\mu_{T'}(v_{j-1}v_{j})=\mu_{T'}(\varphi(v_jv_{j+1}))$ for $2\leq j\leq k-1$.
Therefore, $T$ and $T'$ are $(\varphi,\mu)$-similar with respect to $e_1=v_1v_2$, and $e_2=v_{k-1}v_k=\varphi(e_1)$. Note that
$$\mu_{T}(e_1)=\mu_{T}(v_1v_2)=n_1+1<n_1+k=\mu_{T'}(v_{k-1}v_k)=\mu_{T'}(e_2).$$
By Lemma \ref{lem-5} (1), we have $T\prec T'$.

{\flushleft\bf Case 2.} $T$ has exactly one center edge.

It is clear that $n_k+1=\lceil\frac{n}{2}\rceil$ for $n$ is odd or $n_k+1\leq\frac{n}{2}$ for $n$ is even by Lemma \ref{lem-4} (2) and (3). We may assume that $v_hv_{h+1}$ is the center edge of $T$, where $h\geq 3$ since $n_1+2\leq n_k$. Thus
$$\mu_{T}(v_jv_{j+1})=\left\{\begin{array}{ll}
n_1+j,& \mbox{ if $1\leq j\leq h$}\\
n_k+k-j,& \mbox{ if $h+1\leq j\leq k-1$}\\
\end{array}\right.$$
and
$$\mu_{T'}(v_jv_{j+1})=\left\{\begin{array}{ll}
n_1+j+1,& \mbox{ if $1\leq j\leq h-1$}\\
n_k+k-j-1,& \mbox{ if $h\leq j\leq k-1$.}\\
\end{array}\right.
$$
It follows that $\mu_{T}(v_jv_{j+1})=\mu_{T'}(v_{j-1}v_{j})=\mu_{T'}(\varphi(v_jv_{j+1}))$ for $2\leq j\leq k-1$.
Therefore, $T$ and $T'$ are $(\varphi,\mu)$-similar with respect to $e_1=v_1v_2$, and $e_2=v_{k-1}v_k=\varphi(e_1)$. Since $n_1+2\leq n_k$ we have
$$\mu_{T}(e_1)=\mu_{T}(v_1v_2)=n_1+1<n_k=\mu_{T'}(v_{k-1}v_k)=\mu_{T'}(e_2).$$
By Lemma \ref{lem-5} (1), we have $T\prec T'$.

{\flushleft\bf Case 3.} $T$ has two center edges.

By Lemma \ref{lem-4} (3), we see that $T$ has only one centroidal vertex and $n_k+1\leq\lfloor\frac{n}{2}\rfloor$. Let $v_hv_{h+1}$ and $v_{h+1}v_{h+2}$ be two center edges in $T$, here $h\geq 3$ since $n_1+2\leq n_k$. Thus, it holds that
$$\mu_{T}(v_jv_{j+1})=\left\{\begin{array}{ll}
n_1+j,& \mbox{ if $1\leq j\leq h-1$}\\
n_1+h=\lfloor\frac{n}{2}\rfloor,& \mbox{ if $j=h$}\\
n_k+k-j,& \mbox{ if $h+1\leq j\leq k-1$}\\
\end{array}\right.$$
and
$$\mu_{T'}(v_jv_{j+1})=\left\{\begin{array}{ll}
n_1+j+1,& \mbox{ if $1\leq j\leq h-1$}\\
n_k+k-h-1=\lfloor\frac{n}{2}\rfloor,& \mbox{ if $j=h$}\\
n_k+k-j-1,& \mbox{ if $h+1\leq j\leq k-1$.}\\
\end{array}\right.
$$
Similarly,  $T$ and $T'$ are $(\varphi,\mu)$-similar with respect to $e_1=v_1v_2$, and $e_2=v_{k-1}v_k=\varphi(e_1)$. Since $n_1+2\leq n_k$ we have
$$\mu_{T}(e_1)=\mu_{T}(v_1v_2)=n_1+1<n_k=\mu_{T'}(v_{k-1}v_k)=\mu_{T'}(e_2).$$
By Lemma \ref{lem-5} (1), we have $T\prec T'$.

We complete this proof.
\end{proof}

Given $n>k\ge 2$, let $n=n_1+n_k+k$, the double star pathes $CP_{n;n_1,n_k}$ can be ordered in the following by the Lemma \ref{lem-7}.

\begin{cor}\label{cor-5}
$CP_{n;0,n-k}\prec CP_{n;1,n-k-1}\prec CP_{n;2,n-k-2}\prec\cdots\prec CP_{n;\lfloor \frac{n-k}{2}\rfloor, \lceil \frac{n-k}{2}\rceil}$.
\end{cor}

\subsection{Maximum  graph in   the order set $\langle\mathcal{C}(n,k),\preceq\rangle$}
In this subsection, we will show that the double star path $CP_{n;\lfloor\frac{n-k}{2}\rfloor,\lceil\frac{n-k}{2}\rceil}$ is the maximum graph in  $\langle\mathcal{C}(n,k),\preceq\rangle$.
\begin{thm}\label{thm-5}
Let $T=CP(n;n_1,\ldots,n_k)\in \mathcal{C}(n,k)$ be the caterpillar tree with respect to the path $P_k=v_1\cdots v_k$, where $2\leq k\leq n-1$. Then
$T\preceq CP_{n;\lfloor\frac{n-k}{2}\rfloor,\lceil\frac{n-k}{2}\rceil}$,
with equality if and only if $T= CP_{n;\lfloor\frac{n-k}{2}\rfloor,\lceil\frac{n-k}{2}\rceil}$.
\end{thm}
\begin{proof}
Let $T^*=CP_{n;\lfloor\frac{n-k}{2}\rfloor,\lceil\frac{n-k}{2}\rceil}$,  $T$ be a the maximum graph in $\langle\mathcal{C}(n,k),\preceq\rangle$ and   we will   show that $T= T^*$. By Lemma \ref{lem-7}, it suffices to show that $T$ has no any  branching vertex $v_i$ as the   internal vertex of $P_k$. By the way of  contradiction,  we assume  that $v_i$ is a branching vertex, i.e., $n_i>0$, for some $1<i<k$. We distinguish two situations bellow.

{\flushleft\bf Case 1.} $T$ has center edge.

Let  $e_h=v_hv_{h+1}$ be a center edge of $T$, where $1\le h\le k-1$. Then $\mu(e_h)=\min\{n_{v_h}(e_h),n_{v_{h+1}}(e_h)\}=\lfloor\frac{n}{2}\rfloor$. We may assume that $n_{v_h}(e_h)=\lfloor\frac{n}{2}\rfloor$ and $n_{v_{h+1}}(e_h)=\lceil\frac{n}{2}\rceil$ by symmetry of $P_k$.  Let $T'$ be a branch-shift of $T$ from $v_i$ to $v_{i'}$, where $i'=i+1$ or $i-1$. Let $U(v_i)=\{x_1,\ldots,x_{n_i}\}$ be the pendent vertices sticking at $v_i$. Now we  define a bijection $\varphi: E(T)\longrightarrow E(T')$ such that
\begin{equation}\nonumber \label{simi-eq-11}\varphi(e)=\left\{\begin{array}{ll}
e, & \mbox{ if $e\not=v_ix_l$ for any $ x_l\in U(v_i)$ }\\
v_{i'}x_l, &  \mbox{ if $e=v_ix_l$ for $ x_l\in U(v_i)$. }
\end{array}\right.
\end{equation}
It is clear  that   $\mu_{T}(e)=\mu_{T'}(\varphi(e))$ if $e\not=v_iv_{i'}$, and $\varphi(v_iv_{i'})=v_iv_{i'}$. Therefore, $T$ and $T'$ are $(\varphi,\mu)$-similar with respect to $e_i=v_iv_{i'}$. If $i\le h$, we can take $i'=i-1$ and have $\mu_{T}(e_i)=n_{v_{i-1}}(e_i)<n_{v_{i-1}}(e_i)+n_{i}=\mu_{T'}(e_i)<\lfloor\frac{n}{2}\rfloor$. By Lemma \ref{lem-5} (1), we have $T\prec T'$ which contradicts the maximum hypothesis
of $T$. If $i\ge h+1$, we take $i'=i+1$. Clearly,  if $i>h+1$ then $\mu_{T}(e_i)=n_{v_{i+1}}(e_i)<n_{v_{i+1}}(e_i)+n_{i}=\mu_{T'}(e_i)<\lfloor\frac{n}{2}\rfloor$; if  $i=h+1$, then
$n_{v_{h+2}}(e_{h+1})=n_{v_{h+2}}(v_{h+1}v_{h+2})=n_{v_{h+1}}(e_{h})-1-n_{h+1}=\lceil\frac{n}{2}\rceil-1-n_{h+1}$, and we also have $\mu_{T}(e_{h+1})=n_{v_{h+2}}(e_{h+1})<n_{v_{h+2}}(e_{h+1})+n_{h+1}=\mu_{T'}(e_{h+1})\le\lfloor\frac{n}{2}\rfloor$. Thus $T\prec T'$ by Lemma \ref{lem-5} (1), a contradiction.

{\flushleft\bf Case 2.} $T$ has no center edge.

According to Lemma \ref{lem-4} (2), $T$ has only one proper centroidal vertex $v_h$, where $1\le h\le k$. Then $n_{v_{h+1}}(v_hv_{h+1}), n_{v_{h-1}}(v_{h-1}v_h)<\lfloor\frac{n}{2}\rfloor$. We also consider $T'$ defined in Case 1. If $i<h$, we can  take $i'=i-1$ and have $\mu_{T}(e_i)=n_{v_{i-1}}(e_i)<n_{v_{i-1}}(e_i)+n_{i}=\mu_{T'}(e_i)<\lfloor\frac{n}{2}\rfloor$; if $i> h$, we can take $i'=i+1$ and have $\mu_{T}(e_i)=n_{v_{i+1}}(e_i)<n_{v_{i+1}}(e_i)+n_{i}=\mu_{T'}(e_i)<\lfloor\frac{n}{2}\rfloor$.  It follows that   $T\prec T'$ as in Case 1. If $i=h$, we can take $T''$ that is obtained from $T$ by shifting exactly one  edge  $v_hx_1$ to $v_{h-1}x_1$, and define
$$\psi(e)=\left\{\begin{array}{ll}
e, & \mbox{ if $e\not=v_hx_1$ }\\
v_{h-1}x_1, &  \mbox{ if $e=v_hx_1$. }
\end{array}\right.
$$
Clearly, $T$ and $T''$ are $(\psi,\mu)$-similar with respect to $e_{h-1}=v_{h-1}v_h$. We have $\mu_{T}(e_{h-1})=n_{v_{h-1}}(e_{h-1})<n_{v_{h-1}}(e_{h-1})+1=\mu_{T''}(e_{h-1})\le\lfloor\frac{n}{2}\rfloor$.  By Lemma \ref{lem-5} (1), we have $T\prec T''$, a contradiction.

We complete this proof.
\end{proof}

\subsection{Minimum graph in the order set $\langle\mathcal{C}(n,k),\preceq\rangle$}

For $1\le s\le k$, let
$CP^s_{n,k}=CP(n;\underbrace{0,\ldots,0}_{s-1},n-k,\underbrace{0,\ldots,0}_{k-s})$ denote the caterpillar tree shown in Figure \ref{fig-4}. We always assume that $s\leq \lceil\frac{k}{2}\rceil$ because of
$CP^s_{n,k}= CP^{k-s+1}_{n,k}$ by symmetry.

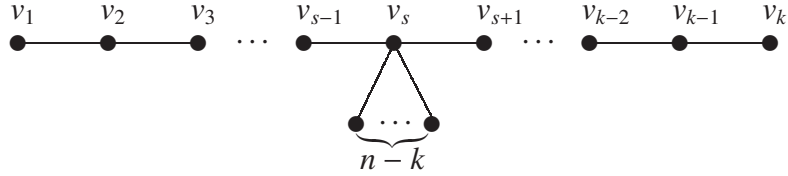
\begin{figure}[H]
\unitlength 1mm 
\linethickness{0.4pt}
\ifx\plotpoint\undefined\newsavebox{\plotpoint}\fi 
\begin{picture}(102,23)(0,5)
\put(2,21.75){\line(1,0){24}}
\put(40,21.75){\line(1,0){24}}
\put(78,21.75){\line(1,0){24}}
\put(2,21.75){\circle*{2}}
\put(14,21.75){\circle*{2}}
\put(26,21.75){\circle*{2}}
\put(40,21.75){\circle*{2}}
\put(52,21.75){\circle*{2}}
\put(64,21.75){\circle*{2}}
\put(78,21.75){\circle*{2}}
\put(90,21.75){\circle*{2}}
\put(102,21.75){\circle*{2}}

\multiput(52,21.75)(-.0336879433,-.070035461){150}{\line(0,-1){.070035461}}

\multiput(52,21.75)(.0336363636,-.0654545455){150}{\line(0,-1){.0654545455}}

\put(47,11){\circle*{2}}
\put(57,11){\circle*{2}}
\put(1,25){$v_{1}$}
\put(13,25){$v_{2}$}
\put(25,25){$v_{3}$}
\put(39,25){$v_{s-1}$}
\put(51,25){$v_{s}$}
\put(63,25){$v_{s+1}$}
\put(77,25){$v_{k-2}$}
\put(89,25){$v_{k-1}$}
\put(101,25){$v_{k}$}
\put(47,10){$\underbrace{}$}
\put(47.5,5){$n-k$}
\put(50,11){\ldots}
\put(31,21.75){\ldots}
\put(69,21.75){\ldots}
\end{picture}
  \caption{ The graph $CP_{n,k}^s$}\label{fig-4}
\end{figure}

Notice that $CP_{n,1}^1=CP(n;n-1)$ and $CP_{n,2}^1=CP(n;n-2,0)$ are the star $K_{1,n-1}$, in general we always assume that $3\le k \le n-1$. In this subsection, we will prove that the caterpillar tree $CP_{n,k}^{\lceil\frac{k}{2}\rceil}$ is the minimum graph in $\langle\mathcal{C}(n,k),\preceq\rangle$.

\begin{lem}\label{lem-6}
Let $CP_{n,k}^s$ be the caterpillar tree with respect to $P_k=v_1v_2\cdots v_k$, where $k\geq3$, $s<\frac{k}{2}$ and $n-k\geq1$, then $CP_{n,k}^s\succ CP_{n,k}^{s+1}$.
\end{lem}
\begin{proof}
Let $T=CP_{n,k}^s$, and $T'=CP_{n,k}^{s+1}$ is just obtained from $T$ by shifting  branch $T_{v_s}$ from $v_s$ to $v_{s+1}$. Let $U(v_s)=\{x_1,\ldots,x_{n-k}\}$ be the pendent vertices sticking at $v_s$. Now we define a bijection $\varphi: E(T)\longrightarrow E(T')$  such that
\begin{equation}\nonumber\label{simi-eq-1}\varphi(e)=\left\{\begin{array}{ll}
e, & \mbox{ if $e\not=v_sx_i$ for any $ x_i\in U(v_s)$ }\\
v_{s+1}x_i, &  \mbox{ if $e=v_sx_i$ for $ x_i\in U(v_s)$. }
\end{array}\right.
\end{equation}
We know that  $T$ and $T'$ are $(\varphi,\mu)$-similar with respect to $e_s=v_sv_{s+1}$, and $e_s=\varphi(e_s)$. Also note that $s<\frac{k}{2}$ and $n-k\geq1$,  for the edge $e_s=v_sv_{s+1}$ we have
$$\begin{array}{ll}\mu_{T}(e_s)&=\min\{s+n-k,k-s\}\\
&>s\\
&=min\{s,n-s\}\\
&=\mu_{T'}(e_s).
\end{array}
$$
It follows $CP_{n,k}^s\succ CP_{n,k}^{s+1}$ by the Lemma \ref{lem-5} (2).
\end{proof}

\begin{thm}\label{thm-4}
Let $T=CP(n;n_1,\ldots,n_k)\in \mathcal{C}(n,k)$ be the caterpillar tree respect to the path $P_k=v_1\cdots v_k$, where $2\leq k\leq n-1$. Then $T\succeq CP_{n,k}^{\lceil\frac{k}{2}\rceil}$, with equality if and only if $T= CP_{n,k}^{\lceil\frac{k}{2}\rceil}$.
\end{thm}
\begin{proof}
If $k=2$, then $T=CP(n;n_1,n_2)$. Since $ CP(n;0,n-2)= CP_{n,2}^1$  is also a star $S_n$, we have  $CP(n;n_1,n_2)\succeq CP(n;0,n-2)=CP_{n,2}^1=S_n$ by the Corollary \ref{cor-5}. In  what follows, we assume that  $k\geq 3$ and distinguish three cases with regard to the number of center edges in $T$.

Let $T^*=CP_{n,k}^{\lceil\frac{k}{2}\rceil}$, $T$ be a minimum graph in $\langle\mathcal{C}(n,k),\preceq\rangle$ and   we will show that $T= T^*$. By Lemma \ref{lem-6}, it suffices to show that $T$ has exactly one  branching vertex $v_i$ as the internal vertex of $P_k$. By the way of  contradiction,  we assume  that $v_l$ and $v_t$ are branching vertices, i.e., $n_l,n_t>0$, for some $1\leq l<t\leq k$. We distinguish two situations bellow.

{\flushleft\bf Case 1.} $T$ has a center edge.

Let  $e_h=v_hv_{h+1}$ be a center edge of $T$, where $1\le h\le k-1$. Then $\mu(e_h)=\min\{n_{v_h}(e_h),n_{v_{h+1}}(e_h)\}=\lfloor\frac{n}{2}\rfloor$. We may assume that $n_{v_h}(e_h)=\lfloor\frac{n}{2}\rfloor$ and $n_{v_{h+1}}(e_h)=\lceil\frac{n}{2}\rceil$ by symmetry of $P_k$.
First assume that  $l=h$ and $t=h+1$.   Let $T'$ be a branch-shift of $T$ from $v_h$ to $v_{h+1}$. Let $U(v_h)=\{x_1,\ldots,x_{n_h}\}$ be the pendent vertices sticking at $v_h$. Now we  define a bijection $\varphi: E(T)\longrightarrow E(T')$ such that
\begin{equation}\nonumber\varphi(e)=\left\{\begin{array}{ll}
e, & \mbox{ if $e\not=v_hx_i$ for any $ x_i\in U(v_h)$ }\\
v_{h+1}x_i, &  \mbox{ if $e=v_hx_i$ for $ x_i\in U(v_h)$. }
\end{array}\right.
\end{equation}
It is clear  that   $\mu_{T}(e)=\mu_{T'}(\varphi(e))$ if $e\not=v_hv_{h+1}$, and $\varphi(v_hv_{h+1})=v_hv_{h+1}$. Therefore, $T$ and $T'$ are $(\varphi,\mu)$-similar with respect to $e_h=v_hv_{h+1}$. We have $\lfloor\frac{n}{2}\rfloor=\mu_{T}(e_h)=n_{v_{h}}(e_h)>n_{v_{h}}(e_h)-n_{h}=\mu_{T'}(e_h)$. By Lemma \ref{lem-5} (2), we have $T\succ T'$ which contradicts the minimum hypothesis
of $T$. Next assume that $l<h$ or $t>h+1$. Without loss of generality, let $l<h$ by symmetry of $P_k$. Let $T''$ be a branch-shift of $T$ from $v_l$ to $v_{l+1}$ and define a bijection $\psi: E(T)\longrightarrow E(T'')$ such that
\begin{equation}\label{psi-eq-1}\psi(e)=\left\{\begin{array}{ll}
e, & \mbox{ if $e\not=v_lx_i$ for any $ x_i\in U(v_l)$ }\\
v_{l+1}x_i, &  \mbox{ if $e=v_lx_i$ for $ x_i\in U(v_l)$. }
\end{array}\right.
\end{equation}
It is clear  that   $\mu_{T}(e)=\mu_{T''}(\psi(e))$ if $e\not=v_lv_{l+1}$, and $\varphi(v_lv_{l+1})=v_lv_{l+1}$. Therefore, $T$ and $T''$ are $(\psi,\mu)$-similar with respect to $e_l=v_lv_{l+1}$. We have $\lfloor\frac{n}{2}\rfloor>\mu_{T}(e_l)=n_{v_{l}}(e_l)>n_{v_{l}}(e_l)-n_{l}=\mu_{T''}(e_l)$. By Lemma \ref{lem-5} (2), we have $T\succ T''$, a contradiction again.

{\flushleft\bf Case 2.}  $T$ has no center edge.

According to  Lemma \ref{lem-4} (2), $T$ has only one proper centroidal vertex $v_h$. Hence $n_{v_{h+1}}(v_hv_{h+1}), n_{v_{h-1}}(v_{h-1}v_h)<\lfloor\frac{n}{2}\rfloor$. Then  at least one of $l<h$ and $t>h$ will occur. We may assume that $l<h$ by symmetry of $P_k$.  Again let $T''$ be a branch-shift of $T$ from $v_l$ to $v_{l+1}$ and define $\psi$ as in Eq. (\ref{psi-eq-1}),  we have $T\succ T''$, a contradiction.

We complete this proof.
\end{proof}

From Theorem \ref{thm-5} and Theorem \ref{thm-4}, we obtain the following result.

\begin{cor}\label{cor-4-2}
Let $T=CP(n;n_1,\ldots,n_k)\in \mathcal{C}(n,k)$. Then
$CP_{n,k}^{\lceil\frac{k}{2}\rceil}\preceq T\preceq CP_{n;\lfloor\frac{n-k}{2}\rfloor,\lceil\frac{n-k}{2}\rceil}$, where $2\leq k\leq n-1$.
\end{cor}
Denote by $\mathcal{T}(n,k-1)$ the set of trees of order $n$ in which each tree has diameter $k-1$. In next section we will consider the extremal graph in $\langle\mathcal{T}(n,k-1),\preceq\rangle$.

\section{The minimum graph in the order set $\langle\mathcal{T}(n,k-1),\preceq\rangle$}
Theorem \ref{thm-4} conforms that  $CP_{n,k}^{\lceil\frac{k}{2}\rceil}$ is  the minimum graph among $\langle\mathcal{C}(n,k),\preceq\rangle$, and $CP_{n,k}^{\lceil\frac{k}{2}\rceil}\in \mathcal{T}(n,k-1)$ for $k\geq3$. To prove that $CP_{n,k}^{\lceil\frac{k}{2}\rceil}$ is also the minimum graph among $\langle\mathcal{T}(n,k-1),\preceq\rangle$, we need the following edge-moving transformation.
Let $T_1$ and $T_2$ be two trees of order $n_1\geq2$ and $n_2\geq2$, respectively, $T$ be the tree obtained from $T_1$ and $T_2$ by adding an edge between a vertex $u$ of $T_1$ and a vertex $v$ of $T_2$, and  $T'$ be the tree obtained from $T$ by contracting $uv$ to a vertex $u$ and attaching a pendent vertex $v$ to $u$ (shown in Figure \ref{fig-5}). We call $T'$ be the edge-moving transformation of $T$ with respect to $uv$.

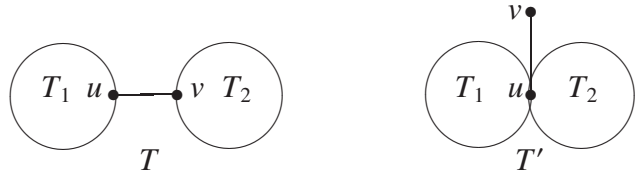
\begin{figure}[h]
\unitlength 1mm 
\linethickness{0.4pt}
\ifx\plotpoint\undefined\newsavebox{\plotpoint}\fi 
\begin{picture}(82.732,21.121)(0,0)
\put(6.915,8.77){\circle{13.83}}
\put(28.992,8.847){\circle{13.83}}
\put(62.142,8.996){\circle{13.83}}
\put(75.817,8.847){\circle{13.83}}
\multiput(13.753,8.919)(1.69462,.02973){5}{\line(1,0){1.69462}}
\put(69.051,9.068){\line(0,1){11.446}}
\put(13.604,9){\circle*{1.5}}
\put(22.077,9){\circle*{1.5}}
\put(69.051,9){\circle*{1.5}}
\put(69.051,20){\circle*{1.5}}
\put(4,8.7){$T_1$}
\put(10,8.7){$u$}
\put(28,8.7){$T_2$}
\put(23.8,8.7){$v$}
\put(59,8.7){$T_1$}
\put(66,8.7){$u$}
\put(74,8.7){$T_2$}
\put(66,19){$v$}
\put(17,-1){$T$}
\put(67,-1){$T'$}
\end{picture}
  \caption{The graphs $T$ and $T'$}\label{fig-5}
\end{figure}

\begin{lem}\label{lem-8}
Let $T$ be a tree with order $n\geq 4$. If $T'$ is the edge-moving transformation of $T$ with respect to $uv$, then $T\succ T'$.
\end{lem}
\begin{proof}
We define $\varphi: E(T)\longrightarrow E(T')$ such that $\varphi(e)=e$ for any edge $e\in E(T)$. It is clear that $\mu_{T}(e)=\mu_{T'}(\varphi(e))$ if $e\neq uv$. Therefore, $T$ and $T'$ are  $(\varphi,\mu)$-similar with respect to $e_1=uv$.
$$
\mu_{T}(e_1)=\mu_{T}(uv)=min\{|T_1|,|T_2|\}>1=\mu_{T'}(uv)=\mu_{T'}(e_1).
$$
It follows $T\succ T'$ by Lemma \ref{lem-5} (2).
\end{proof}

For $T\in\mathcal{T}(n,k-1)$, $T$ can be presented  by $T=T(n;0,n_2,\ldots,n_{k-1},0)\in \mathcal{T}_{n;0,n_2,\ldots,n_{k-1},0}$, i.e., $T=T(n;0,n_2,\ldots,n_{k-1},0)$ is a tree with respect to $P_{k}=v_1v_2\cdots v_k$ such that $|T_{v_i}|=n_i$  for $i=2,3,\ldots,k-1$. By edge-moving transformation, we can transfer  the  edges of  $T_{v_i}$, one by one, to pendent  edges sticking at $v_i$, and finally we get the caterpillar tree $CP(n;0,n_2,\ldots,n_{k-1},0)$ from $T(n;0,n_2,\ldots,n_{k-1},0)$. Therefore, by applying Lemma \ref{lem-8}, we have $CP(n;0,n_2,\ldots,n_{k-1},0)\preceq  T(n;0,n_2,\ldots,n_{k-1},0)$. By applying Theorem \ref{thm-4}, we have $CP_{n,k}^{\lceil\frac{k}{2}\rceil}\preceq CP(n;0,n_2,\ldots,n_{k-1},0)$, and the equality holds if and only if  $T= CP_{n,k}^{\lceil\frac{k}{2}\rceil}$. It follows the result.
\begin{thm}\label{CP-thm-1}
For $T\in \mathcal{T}(n,k-1)$, $T\succeq CP_{n,k}^{\lceil\frac{k}{2}\rceil}$ with equality holds if  and only if  $T= CP_{n,k}^{\lceil\frac{k}{2}\rceil}$.
\end{thm}

By repeating applying Lemma \ref{lem-8}, we get the order of minimum graphs with  different diameters but the same order $n$ as follows:
\begin{cor}\label{cor-5-1}
$S_n= CP_{n,2}^1= CP_{n,3}^2\prec CP_{n,4}^2 \prec\cdots \prec CP_{n,n-2}^{\lceil\frac{n-2}{2}\rceil}\prec CP_{n,n-1}^{\lceil\frac{n-1}{2}\rceil}\prec P_n$.
\end{cor}
By Theorem \ref{thm-5} and Corollary \ref{cor-5-1}, we have the known result in \cite{Vuki}.

\begin{cor}\label{cor-5-2}
For $T\in \mathcal{T}_n$, we have $S_n\preceq T\preceq P_n$.
\end{cor}
\begin{remark}
We know from Theorem \ref{thm-4} and Theorem \ref{CP-thm-1} that $CP_{n,k}^{\lceil\frac{k}{2}\rceil}$ is the common  minimum graph in $\langle\mathcal{C}(n,k),\preceq\rangle$ and  $\langle\mathcal{T}(n,k-1),\preceq\rangle$, respectively. From Corollary  \ref{cor-4-2}, $CP_{n;\lfloor\frac{n-k+2}{2}\rfloor,\lceil\frac{n-k+2}{2}\rceil}$ is the maximum graph in $\langle\mathcal{C}(n,k-2),\preceq\rangle$ with diameter $k-1$. However, the maximum graph in $\langle\mathcal{T}(n,k-1),\preceq\rangle$ leaves unknown.
\end{remark}

\section{Extremal  graphs in the order set $\langle\mathcal{T}_{n}(q),\preceq\rangle$}

Let $\mathcal{T}_{n}(q)$ be a set of trees with order $n$ and $q$ pendent vertices. In this section, we will give the maximum and minimum graphs in $\langle\mathcal{T}_{n}(q),\preceq\rangle$. If $q=2$, then $\mathcal{T}_{n}(2)$ contains a unique graph $P_n=CP_{n;1,1}$. Similarly $S_n=CP_{n;\lfloor\frac{n-1}{2}\rfloor,\lceil\frac{n-1}{2}\rceil}$ is a unique graph in $\mathcal{T}_{n}(n-1)$. Thus we always assume that $2<q<n-1$ in this section.

\subsection{ Maximum graph in the order set $\langle\mathcal{T}_{n}(q),\preceq\rangle$ }

Recall that  $CP_{n;\lfloor\frac{q}{2}\rfloor,\lceil\frac{q}{2}\rceil}$ is the double star path  obtained from the path $P_{n-q}=v_1v_2\cdots v_{n-q}$ by averagely  sticking $q$ pendent vertices at $v_1$ and $v_k$ (see Figure \ref{fig-3}), respectively. It is clear that $CP_{n;\lfloor\frac{q}{2}\rfloor,\lceil\frac{q}{2}\rceil}\in \mathcal{T}_{n}(q)$. In this section,  we will show that $CP_{n;\lfloor\frac{q}{2}\rfloor,\lceil\frac{q}{2}\rceil}$ is the unique maximum graph in $\langle\mathcal{T}_{n}(q),\preceq\rangle$.

\begin{thm}\label{qq-thm-1}
Let $T\in \mathcal{T}_{n}(q)$, where $2<q<n-1$. Then
$T\preceq CP_{n;\lfloor\frac{q}{2}\rfloor,\lceil\frac{q}{2}\rceil}$,
with equality if and only if $T= CP_{n;\lfloor\frac{q}{2}\rfloor,\lceil\frac{q}{2}\rceil}$.
\end{thm}
\begin{proof}
If $q=n-2$, it is clear that $T\preceq CP_{n;\lfloor\frac{n-2}{2}\rfloor,\lceil\frac{n-2}{2}\rceil}$ by  Lemma \ref{lem-7}. In  what follows, we assume that  $q\leq n-3$.
Let $T$ be a maximum graph in $\langle\mathcal{T}_{n}(q),\preceq\rangle$. We will show that $T=CP_{n;\lfloor\frac{q}{2}\rfloor,\lceil\frac{q}{2}\rceil}$ in what follows. According to Lemma \ref{lem-2} (1), we may assume that $u$ is a  centroidal vertex of $T$ with neighbors $v_1,\ldots,v_r$.
For an edge $e_i=uv_i$, let  $T_i=T_{v_i}(e_i)$ denote the  component of $T-e_i$ containing $v_i$. Without loss of generality, let $m_i$ be the order of $T_i$ such that  $1\leq m_1\leq m_2\leq\cdots\leq m_r$, $q_i$ be the number of pendent vertices of $T$ in $T_i$ for $i=1,2,\ldots,r$,  and thus $q=\sum_{i=1}^rq_i$.
It means that each $T_i$ has $m_i$ vertices and contains $q_i$ pendent vertices of $T$. Since $q\leq n-3$, there exists $m_t>1$ for some $1\le t\le r$. Denote by $S_{m_t}(q_t)=CP_{m_t;0,q_t}$ the tree  consisting of a path $P_{m_t-q_t}=z_1\cdots z_{m_t-q_t}$ and sticking  $q_t$ pendent vertices at  $z_{m_t-q_t}$, where $z_1=v_t$.

First we  show that $T_t=S_{m_t}(q_t)$. In fact, it suffices to prove that $T_t$ has at most one vertex of degree great than two in $T$. Otherwise, there exists two adjacent vertices $x_1$ and $x_1'$  of $T_t$ with $d_T(x_1)\ge 3$ and $d_T(x_1')\ge 2$ such that $H=v_t\cdots x_1x_1'$ is a path in $T_t$, where  $x_1$ may be equal to $v_t$. Let $y_1\in N(x_1)\backslash H$, and $T_{y_1}(x_1y_1)$ be the component of $T-x_1y_1$ containing $y_1$.  We construct $T'$ obtained from $T$ by deleting edge  $x_1y_1 $ and adding $x_1'y_1$. Clearly, $T'\in \mathcal{T}_{n}(q)$. We define $\varphi: E(T)\longrightarrow E(T')$ such that
\begin{equation}\nonumber\label{simi-eq-4}\varphi(e)=\left\{\begin{array}{ll}
e, & \mbox{ if $e\not=x_1y_1$ is any edge of $T$}\\
x_1'y_1, &  \mbox{ if $e=x_1y_1$. }
\end{array}\right.
\end{equation}
It is clear that $\mu_{T}(e)=\mu_{T'}(\varphi(e))$ if $e\neq x_1x_1'$. Therefore, $T$ and $T'$ are  $(\varphi,\mu)$-similar with respect to $e_1=x_1x_1'$, and $\varphi(e_1)=e_1$. Additionally, since $u$ is a  centroidal vertex of $T$, we have
$$
\mu_{T'}(e_1)=\mu_{T}(x_1'x_1)+|T_{y_1}(x_1y_1)|>\mu_{T}(x_1'x_1)=\mu_{T}(e_1).
$$ By Lemma \ref{lem-5} (1), we have $T\prec T'$.
However, $T'\in \mathcal{T}_{n}(q)$, this is a contradiction. It implies that  $T_t$ is either a path in the case of $q_t=1$, or a path  sticking   $q_t$ pendent vertices at its one end, i.e., $T_t= S_{m_t}(q_t)$ in the case of $q_t>1$.

Next  we show  that $T_i\not=v_i$ for any $1\le i\le r$. Since  otherwise, $T_1=v_1$, i.e., $uv_1$ is a pendent edge. Note that $n>q+1$, there is $v_i$ such that $d_{T}(v_i)\ge 2$. If such a $v_i$ is unique then $v_i=v_r$ according to our assumption. Thus $T_r=S_{m_r}(q_r)$ and $T_i=v_i$ for $i=1,2,\ldots,r-1$. It means that $T$ is a caterpillar tree. By Theorem \ref{thm-5}, $T= CP_{n;\lfloor\frac{q}{2}\rfloor,\lceil\frac{q}{2}\rceil}$ as our required.   We may further assume that $d_T(v_{r-1}),d_T(v_r)\ge 2$, and  thus $|T_{r-1}|+|T_r|\le n-2$. Without loss of generality, let $|T_r|<\lfloor\frac{n}{2}\rfloor$.  We can  construct $T''=T-uv_1+v_rv_1\in\mathcal{T}_{n}(q)$ and define $\varphi: E(T)\longrightarrow E(T'')$ such that
\begin{equation}\nonumber\label{simi-eq-4}\varphi(e)=\left\{\begin{array}{ll}
e, & \mbox{ if $e\not=uv_1$ is any edge of $T$}\\
v_rv_1, &  \mbox{ if $e=uv_1$. }
\end{array}\right.
\end{equation}
It is clear that $\mu_{T}(e)=\mu_{T''}(\varphi(e))$ if $e\neq uv_r$. Therefore, $T$ and $T''$ are  $(\varphi,\mu)$-similar with respect to $e_1=uv_r$, and $\varphi(e_1)=e_1$. Additionally, since $u$ is a  centroidal vertex of $T$, we have
$$
\mu_{T''}(e_1)=\min\{|T_r|+1,n-(|T_r|+1)\}=|T_r|+1=\mu_{T}(uv_r)+1>\mu_{T}(uv_r)=\mu_{T}(e_1).
$$
Therefore, $T\prec T''$ by Lemma \ref{lem-5} (1). This contradicts  the maximum hypothesis  of $T$.

At last we show that  $r=2$. Then $T= CP_{n;\lfloor\frac{q}{2}\rfloor,\lceil\frac{q}{2}\rceil}$ by Lemma \ref{lem-7},  it is all right. Otherwise, let $r\ge3$.
Let $T'''$ be the tree obtained from $T$ by deleting the edge $uv_r$ and adding the edge $v_1v_r$. Clearly, $T'''\in \mathcal{T}_{n}(q)$.
Now, we define bijection $\varphi: E(T)\longrightarrow E(T''')$ such that
\begin{equation}\nonumber\label{simi-eq-4}\varphi(e)=\left\{\begin{array}{ll}
e, & \mbox{ if $e\not=uv_r$ is any edge of $T$}\\
v_1 v_r, &  \mbox{ if $e=uv_r$. }
\end{array}\right.
\end{equation}
It is clear that $\mu_{T}(e)=\mu_{T'''}(\varphi(e))$ if $e\neq uv_1$. Therefore, $T$ and $T'''$ are $(\varphi,\mu)$-similar with respect to $e_1=uv_1$, and $\varphi(e_1)=e_1$. Recall that  $|T_1|=\min\{|T_i|, i=1,2,\ldots,r\}$, we have
$$\mu_{T'''}(e_1)=min\{|T_r|+|T_1|,\sum_{i=2}^{r-1}|T_i|+1\}>|T_1|=\mu_{T}(e_1).$$
Therefore,  $T\prec T'''$ by Lemma \ref{lem-5} (1), a contradiction  again.

We complete this proof.
\end{proof}

Theorem \ref{qq-thm-1} concludes that $CP_{n;\lfloor\frac{q}{2}\rfloor,\lceil\frac{q}{2}\rceil}$ is the maximum graph in $\langle\mathcal{T}_{n}(q),\preceq\rangle$ for $2<q<n-1$.  By Lemma \ref{lem-8}, these maximum graphs can be  ordered in the following corollary.

\begin{cor}\label{cor-3}
$S_n= CP_{n;\lfloor\frac{n-1}{2}\rfloor,\lceil\frac{n-1}{2}\rceil}\prec CP_{n;\lfloor\frac{n-2}{2}\rfloor,\lceil\frac{n-2}{2}\rceil}\prec\cdots\prec CP_{n;2,2}\prec CP_{n;1,2}\prec CP_{n;1,1}=P_n$.
\end{cor}

\subsection{ Minimum graph in the order set $\langle\mathcal{T}_{n}(q),\preceq\rangle$ }

Let $T$  be the starlike tree with center vertex $u$ and  $v_1,v_2,\ldots,v_q$ be the neighbors of $u$. Then    $P_i=P_{v_i}(v_iu)$ is a path component of $T-v_iu$ containing $v_i$. Such a $T$ we call the balanced starlike tree, denote by $SP_{n,q}$, if $|P_j|-|P_i|\le 1$ for any pair of $1\leq i,j\leq q$. We next prove that $SP_{n,q}$ is the unique minimum graph in $\langle\mathcal{T}_{n}(q),\preceq\rangle$.

\begin{thm}\label{thm-6-2}
Let $T\in \mathcal{T}_{n}(q)$, where $2<q<n-1$. Then
$T\succeq SP_{n,q}$,
with equality if and only if $T= SP_{n,q}$.
\end{thm}
\begin{proof}
Let $T$ be a minimum graph in $\langle\mathcal{T}_{n}(q),\preceq\rangle$.  We will show that $T=SP_{n,q}$ in what follows. We may assume that $u$ is a  centroidal vertex of $T$ by Lemma \ref{lem-2} (1).
We will prove the following Claim 1 and Claim 2, which, put together, will get our proof.

{\flushleft\bf Claim 1.} $T$ is a starlike tree with centre vertex $u$.

It only needs to show that $T$ has the unique branching vertex $u$. By the way of contradiction, let $x_1\neq u\in V(T)$ with $d_T(x_1)\geq3$.
Without loss of generality, we assume that $P=u\cdots x_0x_1x_2$ is a path in $T$, where $x_0$ may be equal to $u$.
We may chose $y_1\in N(x_1)\backslash P$ due to $d_T(x_1)\ge 3$. Let $T_{y_1}(x_1y_1)$ be the component of $T-x_1y_1$ containing $y_1$. We can construct $T'$ obtained from $T$ by deleting edge $x_1y_1$ and adding $x_0y_1$. Clearly, $T'\in \mathcal{T}_{n}(q)$, and  we now define $\varphi: E(T)\longrightarrow E(T')$ such that
\begin{equation}\nonumber\label{simi-eq-4}\varphi(e)=\left\{\begin{array}{ll}
e, & \mbox{ if $e\not=x_1y_1$ is any edge of $T$}\\
x_0y_1, &  \mbox{ if $e=x_1y_1$. }
\end{array}\right.
\end{equation}
It is clear that $\mu_{T}(e)=\mu_{T'}(\varphi(e))$ if $e\neq x_0x_1$. Therefore, $T$ and $T'$ are $(\varphi,\mu)$-similar with respect to $e_1=x_0x_1$, and $\varphi(e_1)=e_1$. Since $u$ is a centroidal vertex, we have
$$
\mu_{T'}(e_1)=\mu_{T}(x_0x_1)-|T_{y_1}(x_1y_1)|<\mu_{T}(x_0x_1)=\mu_{T}(e_1).
$$
Therefore, $T\succ T'$ by Lemma \ref{lem-5} (2). This contradicts the minimum hypothesis of $T$.
{\flushleft\bf Claim 2.} $T= SP_{n,q}$.

According to Claim 1, we may assume that $T$ is the starlike tree with center vertex $u$ and $v_1,v_2,\ldots,v_{q}$ be the neighbors of $u$. Let $P_i=P_{v_i}(v_iu)$ be the component of $T-v_iu$ containing $v_i$. Then each  $P_i$ is a path.

If we show that $|P_j|-|P_i|\le 1$ for $1\leq i,j\leq q$, then $T= SP_{n,q}$, it is all right. Otherwise, without loss of generality, we assume that $|P_{i_0-1}|+2\leq |P_{i_0}|$ for some $1< i_0\le q$.
Now let $T'$ be the tree obtained from $T$ by deleting the edges $uv_i$ and adding the edges $v_{i_0}v_i$ for $i\not=i_0-1,i_0$. Clearly, $T'\in \mathcal{T}_{n}(q)$. We define $\varphi: E(T)\longrightarrow E(T')$ such that
\begin{equation}\nonumber\label{simi-eq-4}\varphi(e)=\left\{\begin{array}{ll}
e, & \mbox{ if $e\not=uv_{i}$ is any edge of $T$ for $i=1,\ldots,i_0-2, i_0+1,\ldots,q$}\\
v_{i_0}v_i, &  \mbox{ if $e=uv_i$ for $i=1,\ldots,i_0-2, i_0+1,\ldots,q$. }
\end{array}\right.
\end{equation}
It is clear that $\mu_{T}(e)=\mu_{T'}(\varphi(e))$ if $e\neq uv_{i_0}$. Therefore, $T$ and $T'$ are $(\varphi,\mu)$-similar with respect to $e_1=uv_{i_0}$, and $\varphi(e_1)=e_1$. Additionally, we have
$$
\mu_{T'}(e_1)=min\{\sum_{i\not=i_0-1}|P_i|,|P_{i_0-1}|+1\}<|P_{i_0}|=\mu_{T}(e_1).
$$
Therefore, $T\succ T'$ by Lemma \ref{lem-5} (2), a contradiction.

We complete this proof.
\end{proof}

Let $\mathcal{T}_{n}^\Delta$ be a set of trees with order $n$ and maximum degree $\Delta\geq3$. In the next section, we will consider the maximal graph in $\langle\mathcal{T}_{n}^\Delta,\preceq\rangle$.

\section{Maximum  graph in the order set $\langle\mathcal{T}_{n}^\Delta,\preceq\rangle$}

In this section, we will give the maximum  graph in the  order set $\langle\mathcal{T}_{n}^\Delta,\preceq\rangle$. If $\Delta=n-1$ then $\mathcal{T}_{n}^\Delta$ contains exactly one $S_n$, which is both maximum and minimum graph in the
order set $\langle\mathcal{T}_{n}^{n-1},\preceq\rangle$. Similarly if $\Delta=2$ then $T$ is a path. Thus we always assume in this section that $2<\Delta<n-1$.

Let $CP_{n;1,\Delta-1}$ be the graph obtained from a path $P_{n-\Delta}=v_1\cdots v_{n-\Delta}$ by adding $\Delta-1$ pendent vertices at $v_{n-\Delta}$ and one pendent vertex at $v_1$. We will prove that $CP_{n;1,\Delta-1}$ is the unique maximum graph in $\langle\mathcal{T}_{n}^\Delta,\preceq\rangle$.

\begin{thm}\label{PM-thm-1}
Let $T\in \mathcal{T}_{n}^\Delta$ where  $2<\Delta<n-1$. Then
$T\preceq CP_{n;1,\Delta-1}$,
with equality if and only if $T= CP_{n;1,\Delta-1}$.
\end{thm}
\begin{proof}
Let $T$ be a maximum graph in $\langle\mathcal{T}_{n}^\Delta,\preceq\rangle$. By Lemma \ref{lem-2} (1), we may assume that $u$ is a  centroidal vertex of $T$ and  show that $T= CP_{n;1,\Delta-1}$. In the following, we need only to show the  Claim 1 and Claim 2, which, put  together, will get our result.

{\flushleft\bf Claim 1.} $T$ is a starlike tree.

First  assume that centroidal vertex $u$ is of maximum vertex  and will show that $T$ is a starlike tree with centre $u$. By the way of contradiction, let $x_1\not=u\in V(T)$ with $d_T(x_1)\geq3$, and
$P=u\cdots x_1$ is a path in $T$. Denote by $D_1(P)$  the set of vertices such that $x_2\in D_1(P)$ iff $P\cdot x_2=u\cdots x_1x_2$ is a path of $T$, by $D_2(P)$ we mean that $x_3\in D_2(P)$ iff $P\cdot x_2x_3=u\cdots x_1x_2x_3$ is a path of $T$ for any $x_3\in D_2(P)$ where $x_2\in D_1(P)$, and $D_k(P)$ is similarly defined for $k\ge3$.

By definition, $D_1(P)\not=\emptyset$. We choose any   $x_2\in D_1(P)$ and  let $P_1=P\cdot x_2=u\cdots x_1x_2$ be a path of $T$. Also we  may select $y_1\in N(x_1)\backslash P_1$ due to $d_T(x_1)\ge 3$. Let $T_{y_1}(x_1y_1)$ be the component of $T-x_1y_1$ containing $y_1$. We can construct $T'$ obtained from $T$ by deleting edge $x_1y_1$ and adding $x_2y_1$. Now we first show that $d_T(x_2)= \Delta$. Since otherwise,  $T'\in \mathcal{T}_{n}^\Delta$, and  we will define $\varphi: E(T)\longrightarrow E(T')$ such that
\begin{equation}\nonumber\label{simi-eq-4}\varphi(e)=\left\{\begin{array}{ll}
e, & \mbox{ if $e\not=x_1y_1$ is any edge of $T$}\\
x_2y_1, &  \mbox{ if $e=x_1y_1$. }
\end{array}\right.
\end{equation}
It is clear that $\mu_{T}(e)=\mu_{T'}(\varphi(e))$ if $e\neq x_1x_2$. Therefore, $T$ and $T'$ are $(\varphi,\mu)$-similar with respect to $e_1=x_1x_2$, and $\varphi(e_1)=e_1$. Additionally, since $u$ is a  centroidal vertex of $T$ on $P$ and $x_1\not=u$, we have
$$
\mu_{T'}(e_1)=\mu_{T}(x_1x_2)+|T_{y_1}(x_1y_1)|>\mu_{T}(x_1x_2)=\mu_{T}(e_1).
$$
Therefore, $T\prec T'$ by Lemma \ref{lem-5} (1). This contradicts  the maximum hypothesis  of $T$. This also  implies that every vertex of $D_1(P)$ has maximum degree $\Delta$ since $x_2$ is an arbitrary vertex in $D_1(P)$. Next we show that $d_T(x_3)=\Delta$ for $x_3\in D_2(P)$.   In fact,  $P_1$ can be extended  as $P_2=u\cdots x_1x_2x_3$ where $x_2\in D_1(P)$. Let $y_2\in N(x_2)\backslash P_2$. We can construct $T''$ obtained from $T$ by deleting $x_2y_2$ and adding $x_3y_2$. Clearly,  if $d_T(x_3)< \Delta$ then $T''\in \mathcal{T}_{n}^\Delta$, and then we define  $\varphi: E(T)\longrightarrow E(T'')$ such that
\begin{equation}\nonumber\label{simi-eq-4}\varphi(e)=\left\{\begin{array}{ll}
e, & \mbox{ if $e\not=x_2y_2$ is any edge of $T$}\\
x_3y_2, &  \mbox{ if $e=x_2y_2$. }
\end{array}\right.
\end{equation}
By regarding $x_2$ as $x_1$ and $y_2$ as $y_1$, we would get  $T\prec T''$ as above arguments, a contradiction again. Similarly  we can show that $d_T(x_4)=\Delta$ for $x_4\in D_3(P)$, and so on. This is impossible since  this procedure can not be terminated and so $T$ is a starlike tree with centre $u$. The above proof also implies that $T$ cannot contain two vertices of degree $\Delta$.

Next assume that $d(u)<\Delta$,  we will show that   $d(u)=2$. Since $T\in \mathcal{T}_n^\Delta$, $T$ has a unique vertex of maximum degree, say $v\not=u$ with $d(v)=\Delta$ according to the above arguments. If there is some $x_1\not=v,u$ such that $3\le d(x_1)< \Delta$, then, as the same arguments, we would get  that $d(x_k)=\Delta$ for $k\ge 2$, where $P_{k-1}=P\cdot  x_2\cdots x_k$ and $x_k\in D_{k-1}(P)$. Thus $T$ is  starlike tree with centre $v$ if $d(u)=2$. Otherwise $3\le d(u)<\Delta$ and $d(x)\leq2$ for any $x\not=v,u$. Thus there exists a path $H=u_1u_2\cdots u_m$ such that $u_t=u$ where $2\le t\le m-1$. Without loss of generality, assume that $t\le \lceil\frac{m}{2}\rceil$. Denote by $e_i=u_iu_{i+1}$ the edges on $H$ for $i=1,2,\ldots,m-1$, we construct a tree  $T^*$  that is obtained from $T$ by deleting edge $e_1=u_1u_2$ and adding a new edge $e_m=u_1u_m$. Now we define $\varphi: E(T)\longrightarrow E(T^*)$ such that
\begin{equation}\nonumber\label{simi-eq-5}\varphi(e)=\left\{\begin{array}{ll}
e, & \mbox{ if $e\not=u_iu_{i+1}$  for $i=1,2,\ldots,m-1$}\\
u_{i+1}u_{i+2}, &  \mbox{ for $e_i=u_iu_{i+1}$ and $i=1,2,\ldots,m-2$ }\\
u_{1}u_{m}, &  \mbox{ for $e_{m-1}=u_{m-1}u_m$.}
\end{array}\right.
\end{equation}
It is easy to verify that $\mu_{T}(e)=\mu_{T'}(\varphi(e))$ if $e\neq u_{t-1}u_{t}$. Therefore, $T$ and $T^*$ are $(\varphi,\mu)$-similar with respect to $e_{t-1}=u_{t-1}u_{t}$, and $\varphi(e_{t-1})=e_{t}$. Recall that $u_t=u$ is centroidal vertex, we have $\mu_T(e_{t-1})=\mu_T(u_{t-1}u_t)=t-1<m-t+1=\mu_{T^*}(u_tu_{t+1})=\mu_{T^*}(e_{t})$,
then $T\prec T^*$ by Lemma \ref{lem-5} (1), a contradiction. Therefore, $d(u)=2$ and $T$ is a starlike tree with centre $v$.

{\flushleft\bf Claim 2.} $T= CP_{n;1,\Delta-1}$.

According to Claim 1, we may assume that $T$ is the starlike tree with centre vertex $v$ and $v_1,v_2,\ldots,v_{\Delta}$ be the neighbors of $v$. Let $P_i=P_{v_i}(v_iv)$ be the component of $T-v_iv$ containing $v_i$. Without loss of generality, we assume that $|P_i|\leq |P_j|$ for $i\leq j$.

If we show that $|P_1|=\cdots=|P_{\Delta-1}|=1$, then $T=CP_{n;1,\Delta-1}$, it is all right. Otherwise, let $|P_{\Delta}|\geq|P_{\Delta-1}|\geq2$.  Let $T'''$ be the tree obtained from $T$ by deleting the edges $vv_i$ and adding the edges $v_{\Delta-1}v_i$ for  $i=1,2,\ldots,\Delta-2$. Clearly, $T'''\in \mathcal{T}_{n}^\Delta$. We define $\varphi: E(T)\longrightarrow E(T''')$ such that
\begin{equation}\nonumber\label{simi-eq-6}\varphi(e)=\left\{\begin{array}{ll}
e, & \mbox{ if $e\not=vv_{i}$ is any edge of $T$ for $i =1,\ldots,\Delta-2$,}\\
v_{\Delta-1}v_i, &  \mbox{ if $e=vv_i$ for $i =1,\ldots,\Delta-2$. }
\end{array}\right.
\end{equation}
It is clear that $\mu_{T}(e)=\mu_{T'''}(\varphi(e))$ if $e\neq vv_{\Delta-1}$. Therefore, $T$ and $T'''$ are $(\varphi,\mu)$-similar with respect to $e_1=vv_{\Delta-1}$, and $\varphi(e_1)=e_1$. Additionally, note that $u$ is centroidal vertex including in $\{v\}\cup V(P_{\Delta})$ due to $|P_{\Delta}|\ge|P_{\Delta-1}|\geq2$,  we have
$$
\mu_{T'''}(e_1)=\min\{\sum_{i=1}^{\Delta-1}|P_i|,|P_{\Delta}|+1\}>|P_{\Delta-1}|=\mu_{T}(e_1).
$$
Therefore, $T\prec T'''$ by Lemma \ref{lem-5} (1), a contradiction.

We complete this proof.
\end{proof}

According to Theorem \ref{PM-thm-1}, given $n$ we know that $T=CP_{n;1,\Delta-1}$ is maximum graph in $\langle\mathcal{T}_{n}^\Delta,\preceq\rangle$ for $2<\Delta<n-1$.  By Lemma \ref{lem-8}, these maximum graphs can be ordered in the following corollary.
\begin{cor}\label{cor-4}
$S_n=CP_{n;1,n-2}\prec CP_{n;1,n-3}\prec\cdots\prec CP_{n;1,3}\prec CP_{n;1,2}\prec CP_{n;1,1}=P_n$.
\end{cor}
\begin{remark}
From Theorem \ref{thm-6-2}, we know that $SP_{n,q}$ is minimum graph
in $\langle\mathcal{T}_{n}(q),\preceq\rangle$. However, the minimum graph  in $\langle\mathcal{T}_{n}^\Delta,\preceq\rangle$ leaves unknown.
\end{remark}

\section{Application}
It is well known that  topological index of a graph is usually defined as a function of distance, such as Wiener index of a graph $G$:  $W(G)=\sum_{\{u,v\}\in V(G)}d(u,v)$. For a pair of vertices $\{u,v\}$ in a tree $T\in \mathcal{T}_n$, there is a unique path $P_{u,v}$ connecting them with exactly $d(u,v)$ edges.  Thus $W(T)$ can be represented by a function of $\mu(e)$ as follows:
$$\begin{array}{ll}W(T)&=\sum_{\{u,v\}\in V(T)}d(u,v)=\sum_{\{u,v\}\in V(T)}|E(P_{u,v})|\\
&=\sum_{e=xy\in E(T)}n_x(e)n_y(e)=\sum_{e=xy\in E(T)}n_x(e)(n-n_x(e))\\
&=\sum_{e\in E(T)}\mu(e)(n-\mu(e)).
\end{array}
$$
Let $\mathbf{r}=(r_1,r_2,\ldots,r_{\lfloor\frac{n}{2}\rfloor})$ be the edge division vector of $T$ and $f(x)=x(n-x)$. As in the proof of Theorem 8 in \cite{Vuki}, $W(T)$ can be further simplified as
\begin{equation}\label{WF-eq-22}W(T)=\sum_{e\in E(T)}\mu(e)(n-\mu(e))=\sum_{1\le i\le \lfloor\frac{n}{2}\rfloor}r_if(i).\end{equation}
Since there are large part of topological indices of a graph can be also  represented by some functions of $\mu(e)$ described as Eq. (\ref{WF-eq-22}), the authors in \cite{Vuki} introduced the notions bellow.
\begin{defi}
Let $F:\mathcal{T}_n\rightarrow \mathbb{R}$ be a topological index and let $f:\mathbb{N}\rightarrow \mathbb{R}$ be a real function defined for positive integers. The  topological index $F$ is an edge additive eccentric topological index if it holds that
$F(G)=\sum_{e\in E(G)}f(\mu(e))$.
Function $f$ is called the edge contribution function of index $F$.
\end{defi}
Let $F$ be the edge additive eccentric topological index and $T,T'\in \mathcal{T}_n$. From Eq. (\ref{WF-eq-22}), the authors in  \cite{Vuki} gave that
\begin{equation}\label{WF-eq-23}
\begin{array}{ll}
F(T')-F(T)=&\sum_{1\le i\le \lfloor\frac{n}{2}\rfloor}(r_i'-r_i)f(1)+\sum_{2\le i\le \lfloor\frac{n}{2}\rfloor}(r_i'-r_i)(f(2)-f(1))+\cdots\\
&+\sum_{\lfloor\frac{n}{2}\rfloor\le i\le \lfloor\frac{n}{2}\rfloor}(r_i'-r_i)(f(\lfloor\frac{n}{2}\rfloor)-f(\lfloor\frac{n}{2}\rfloor-1)),
\end{array}
\end{equation}
Eq. (\ref{WF-eq-23}) implies that $F(T)\le F(T')$ (resp., $F(T)< F(T')$) if $T\prec T'$ and $f(x)$ is increasing (resp., proper increasing). Such an index  $F$ is defined in \cite{Vuki} to be of Wiener type if $f(x)$ is increasing, and of anti-Wiener type if $f(x)$ is decreasing. The above  idea provides of a new method to find the extremal graphs and evolute the bounds with respect to topological index $F$. Summarizing the above arguments leads to the following result.
\begin{thm}\label{FF-thm-1}
Let $\langle\mathcal{H}_n, \preceq\rangle$ be an order subset of $\langle\mathcal{T}_n, \preceq\rangle$ and $H$ be an  minimum (maximum) graph of $\langle\mathcal{H}_n, \preceq\rangle$.  Let $F:\mathcal{H}_n\longrightarrow \mathbb{R}$ be an edge additive eccentric topological index. We have\\
(1) If $F$ is an index of Wiener type, then $H$ is also an  minimum (maximum) graph of $\mathcal{H}_n$ with respect to $F$.\\
(2) If $F$ is an index of anti-Wiener type, then $H$ is maximum (minimum) graph of $\mathcal{H}_n$ with respect to $F$.
\end{thm}
By applying Theorem \ref{FF-thm-1}, the authors in \cite{Vuki} (Theorem 12, 17) proved the  conclusions that we summarize in the Table 1.

\begin{table}[H]\label{tab-1}
\scriptsize
\caption{Some conclusions of indices}
\centering
\begin{tabularx}{420pt}{c|c|c|c|c}
\toprule
Indices & Definition & Edge contribution function & Type & Extremal graphs in $\mathcal{T}_n$ \\
\midrule
\makecell[t]{  Wiener \\ index } &
\makecell[t]{ $W(G)=$\\ $\sum_{\{u,v\}\in V(G)}d(u,v)$} &
\makecell[t]{ $f(x)=x(n-x)$} &
\makecell[t]{ Wiener type} &
\makecell[t]{ $W(S_n)\leq W(T)\leq W(P_n)$}\\
\midrule
\makecell[t]{ Modified \\ Wiener \\  indices } &
\makecell[t]{ \\$^{\lambda}W(G)=$ \\ $\sum_{\{u,v\}\in V(G)}d^{\lambda}(u,v)$} &
\makecell[t]{ \\$f(x)=x^\lambda(n-x)^\lambda$} &
\makecell[t]{ Wiener type for \\$\lambda>0$ and anti-Wiener \\type for $\lambda<0$}&
\makecell[t]{ $^{\lambda}W(S_n)\leq ^{\lambda}W(T)\leq ^{\lambda}W(P_n)$ \\for $\lambda>0$ and \\$^{\lambda}W(P_n)\leq ^{\lambda}W(T)\leq ^{\lambda}W(S_n)$ \\for $\lambda<0$} \\
\midrule
\makecell[t]{ Variable \\ Wiener \\ indices } &
\makecell[t]{ $_{\lambda}W(G)=$ \\$\frac{1}{2}\sum_{e=uv\in E(G)}(n^{\lambda}-$\\${n_u(e)}^{\lambda}-{n_v(e)}^{\lambda})$} &
\makecell[t]{ \\$f(x)=n^\lambda-x^\lambda-(n-x)^\lambda$ } &
\makecell[t]{ Wiener type for \\$\lambda>1$ and anti-Wiener \\type for $\lambda<1$  } &
\makecell[t]{ $_{\lambda}W(S_n)\leq _{\lambda}W(T)\leq _{\lambda}W(P_n)$\\ for $\lambda>1$ and \\$_{\lambda}W(P_n)\leq _{\lambda}W(T)\leq _{\lambda}W(S_n)$ \\for $\lambda<1$ } \\
\midrule
\makecell[t]{ Steiner \\  $k$-Wiener\\  index } &
\makecell[t]{ $SW_k(G)=$ \\ $\sum_{e=uv\in E(G)}\sum_{i=1}^{k-1}$ \\ $\binom{n_u(e)}{i}\binom{n_v(e)}{k-i}$} &
\makecell[t]{ \\$f(x)=\binom{n}{k}-\binom{x}{k}-\binom{n-x}{k}$} &
\makecell[t]{ \\Wiener type } &
\makecell[t]{ \\$SW_k(S_n)\leq SW_k(T)\leq SW_k(P_n)$} \\
\bottomrule
\end{tabularx}
\end{table}

As similar as the above conclusions, we will give the extremal graphs with respect to the  topological indices of  hyper-Wiener index, Wiener-Hosoya index, degree distance, Gutman index and $ABC_2$ index, respectively. It suffices to verify that wether these indices are edge additive eccentric topological index $F(G)$ and wether  the corresponding edge contribution function $f$ is monotonous such that $F(G)=\sum_{e\in E(G)}f(\mu(e))$.

It is easy to verify that  hyper-Wiener index can be written as
$$\begin{array}{ll}
WW(T)
&=\sum_{e=uv\in E(T)}(\frac{1}{2}n_u(e)n_v(e)+\frac{1}{2}n_u(e)^2n_v(e)^2)\\
&=\sum_{e=uv\in E(T)}[\frac{1}{2}n_u(e)(n-n_u(e))+\frac{1}{2}n_u(e)^2(n-n_u(e))^2]\\
&=\sum_{1\le i\le \lfloor\frac{n}{2}\rfloor}r_if_{ww}(i)
\end{array}
$$
where its edge contribution function is $f_{ww}(x)=\frac{1}{2}x(n-x)+\frac{1}{2}x^2(n-x)^2$. It is routine  to verify that $f_{ww}(x)$  is strictly increasing.  Hence the hyper-Wiener index for tree is of Wiener type.

For $e=uv\in E(T)$, note that $h(e)=n_u(e)n_v(e)$ and $h[e]=(n_u(e)-1)(n_v(e)-1)$,  the Wiener-Hosoya index can be written as
$$\begin{array}{ll}
h(T)
&=\sum_{e=uv\in E(T)}[n_u(e)n_v(e)+(n_u(e)-1)(n_v(e)-1)]\\
&=\sum_{e=uv\in E(T)}[n_u(e)(n-n_u(e))+(n_u(e)-1)(n-n_u(e)-1)]\\
&=\sum_{1\le i\le \lfloor\frac{n}{2}\rfloor}r_if_{h}(i)
\end{array}$$
where its edge contribution function is $f_{h}(x)=x(n-x)+(x-1)(n-x-1)$. It is routine  to verify that $f_{h}(x)$  is strictly increasing. Hence the Wiener-Hosoya index for tree is of Wiener type.

The degree distance can be written as
$$\begin{array}{ll}
D'(T)
&=\sum_{e=uv\in E(T)}(4n_u(e)n_v(e)-n)\\
&=\sum_{e=uv\in E(T)}[4n_u(e)(n-n_u(e))-n]\\
&=\sum_{1\le i\le \lfloor\frac{n}{2}\rfloor}r_if_{D'}(i)
\end{array}$$
where its edge contribution function is $f_{D'}(x)=4x(n-x)-n$. It is routine  to verify that $f_{D'}(x)$  is strictly increasing. Hence the degree distance for tree is of Wiener type.

The Gutman index can be written as
$$\begin{array}{ll}
Gut(T)
&=\sum_{e=uv\in E(T)}[4n_u(e)n_v(e)-(2n-1)]\\
&=\sum_{e=uv\in E(T)}[4n_u(e)(n-n_u(e))-(2n-1)]\\
&=\sum_{1\le i\le \lfloor\frac{n}{2}\rfloor}r_if_{Gut}(i)
\end{array}$$
where its edge contribution function is $f_{Gut}(x)=4x(n-x)-(2n-1)$. It is routine to verify that $f_{Gut}(x)$  is strictly increasing. Hence the Gutman index for tree is of Wiener type.

The second atom-bond connectivity index can be written as
$$\begin{array}{ll}
ABC_2(T)
&=\sum_{e=uv\in E(T)}\sqrt{\frac{n-2}{n_u(e)(n-n_u(e))}}\\
&=\sum_{e=uv\in E(T)}\sqrt{n-2}n_u(e)^{-\frac{1}{2}}(n-n_u(e))^{-\frac{1}{2}}\\
&=\sum_{1\le i\le \lfloor\frac{n}{2}\rfloor}r_if_{ABC_2}(i)
\end{array}$$
where its edge contribution function is $f_{ABC_2}(x)=\sqrt{n-2}x^{-\frac{1}{2}}(n-x)^{-\frac{1}{2}}$. It is routine to verify that $f_{ABC_2}(x)$  is strictly decreasing. Hence the second atom-bond connectivity index for tree is of anti-Wiener type.

By setting  $\langle \mathcal{H}_n,\preceq\rangle=\langle\mathcal{C}(n,k),\preceq\rangle, \langle\mathcal{T}(n,k-1),\preceq\rangle,  \langle\mathcal{T}_n(q),\preceq\rangle$ or $\langle\mathcal{T}_{n}^\Delta,\preceq\rangle$, we have determined their extremal graphs in Theorem \ref{thm-5},\ref{thm-4},\ref{CP-thm-1},\ref{qq-thm-1},\ref{thm-6-2} and \ref{PM-thm-1},  respectively. From the above discussions, we get the following theorem by applying Theorem \ref{FF-thm-1}. Its conclusions are summarized in the Table 2.

\begin{thm}\label{thm-9-3}
The extremal graphs in the classes $\mathcal{T}_n$, $\mathcal{C}(n,k)$, $\mathcal{T}(n,k-1)$, $\mathcal{T}_n(q)$ and $\mathcal{T}_{n}^\Delta$ with respect to all the topological indices of $WW(\cdot)$, $h(\cdot)$, $D'(\cdot)$, $Gut(\cdot)$ and $ABC_2(\cdot)$ are listed in Table 2.
\end{thm}
\begin{table}[H]
\footnotesize
\caption{Some conclusions of indices}
\centering
\begin{tabularx}{420pt}{c|c|c|c|c}
\toprule
Type & \multicolumn{2}{|c|}{Wiener type indices} &  \multicolumn{2}{|c|}{anti-Wiener type indices} \\
\midrule
\makecell[t]{Indices   } &
\multicolumn{2}{|c|}{$WW(\cdot)$, $h(\cdot)$, $D'(\cdot)$, $Gut(\cdot)$} &  \multicolumn{2}{|c|}{$ABC_2(\cdot)$}\\
\midrule
\makecell[t]{ Classes } &
\makecell[t]{ The graphs with \\maximum indices } &
\makecell[t]{ The graphs with \\minimum indices } &
\makecell[t]{ The graphs with \\maximum indices} &
\makecell[t]{ The graphs with \\minimum indices}\\
\midrule
\makecell[t]{$\mathcal{T}_n$ } &
\makecell[t]{$P_n$} &
\makecell[t]{$S_n$ } &
\makecell[t]{$S_n$ }&
\makecell[t]{$P_n$ } \\
\midrule
\makecell[t]{$\mathcal{C}(n,k)$  } &
\makecell[t]{$CP_{\lfloor\frac{n-k}{2}\rfloor,\lceil\frac{n-k}{2}\rceil}$ } &
\makecell[t]{$CP_{n,k}^{\lceil\frac{k}{2}\rceil}$ } &
\makecell[t]{$CP_{n,k}^{\lceil\frac{k}{2}\rceil}$ } &
\makecell[t]{$CP_{\lfloor\frac{n-k}{2}\rfloor,\lceil\frac{n-k}{2}\rceil}$ } \\
\midrule
\makecell[t]{$\mathcal{T}(n,k-1)$  } &
\makecell[t]{ } &
\makecell[t]{$CP_{n,k}^{\lceil\frac{k}{2}\rceil}$ } &
\makecell[t]{$CP_{n,k}^{\lceil\frac{k}{2}\rceil}$  } &
\makecell[t]{ } \\
\midrule
\makecell[t]{$\mathcal{T}_n(q)$  } &
\makecell[t]{$CP_{\lfloor\frac{q}{2}\rfloor,\lceil\frac{q}{2}\rceil}$ } &
\makecell[t]{$SP_{n,q}$} &
\makecell[t]{$SP_{n,q}$  } &
\makecell[t]{$CP_{\lfloor\frac{q}{2}\rfloor,\lceil\frac{q}{2}\rceil}$ } \\
\midrule
\makecell[t]{$\mathcal{T}_{n}^\Delta$  } &
\makecell[t]{$CP_{n;1,\Delta-1}$ } &
\makecell[t]{ } &
\makecell[t]{  } &
\makecell[t]{$CP_{n;1,\Delta-1}$ } \\
\bottomrule
\end{tabularx}
\end{table}

There are some sporadic results published in various literatures \cite{Lu,Zhang0,Gutman2,Yu,Cai,Feng,Tomescu,Bonchev,Liu,Wang,Pandey,Nikoli,Liu1,Chen,Zhang2,Vukic} about the topological indices of a graph, such as the indices of  Wiener index, hyper-Wiener index, Wiener-Hosoya index, degree distance, Gutman index and $ABC_2$ index and so on, in which the authors mainly considered the extremal graphs with respect to some  topological indices on subset of trees, such as $\mathcal{T}_{n}$, $\mathcal{C}(n,k)$, $\mathcal\mathcal{T}(n,k-1)$, $\mathcal{T}_n(q)$ and $\mathcal{T}_{n}^\Delta$. We now summarize our results of Theorem \ref{FF-thm-1} and Theorem \ref{thm-9-3} along with the known results in the following Table 3. In this table, each cell lists the extremal graphs of the class specified,  in which the results marked in black are old,  marked blue are new (in Theorem \ref{FF-thm-1} and Theorem \ref{thm-9-3}) and the question mark are still open.

\begin{sidewaystable}
\caption{Some known and new results of indices}
\centering
\begin{tabularx}{650pt}{c|c|c|c|c|c|c}
\toprule
\footnotesize Indices &\footnotesize Type &  \footnotesize trees in $\mathcal{T}_{n}$ &  \footnotesize trees in $\mathcal{C}(n,k)$&  \footnotesize trees in $\mathcal\mathcal{T}(n,k-1)$ &  \footnotesize trees in $\mathcal{T}_n(q)$ &   \footnotesize trees in $\mathcal{T}_{n}^\Delta$ \\
\midrule
\makecell[t]{ \footnotesize Wiener \\ \footnotesize index } &
\makecell[t]{ \footnotesize minimum\\ \footnotesize maximum \\ \footnotesize references} &
\makecell[t]{ \footnotesize $S_n$\\ \footnotesize $P_n$ \\ \footnotesize \cite{Bonchev}} &
\makecell[t]{ \footnotesize $CP_{n,k}^{\lceil\frac{k}{2}\rceil}$ \\ \footnotesize $CP_{\lfloor\frac{n-k}{2}\rfloor,\lceil\frac{n-k}{2}\rceil}$ \\ \footnotesize \cite{Wang}} &
\makecell[t]{ \footnotesize $CP_{n,k}^{\lceil\frac{k}{2}\rceil}$ \\ \textcolor[rgb]{1.00,0.00,0.00}{?} \\ \footnotesize \cite{Liu,Wang}}&
\makecell[t]{ \footnotesize $SP_{n,q}$\\ \footnotesize $CP_{\lfloor\frac{q}{2}\rfloor,\lceil\frac{q}{2}\rceil}$ \\ \footnotesize \cite{Pandey}} &
\makecell[t]{ \textcolor[rgb]{1.00,0.00,0.00}{?} \\ \footnotesize $CP_{n;1,\Delta-1}$ \\ \footnotesize \cite{Wang}}\\
\midrule
\makecell[t]{ \footnotesize Modified \\ \footnotesize Wiener \\ \footnotesize indices } &
\makecell[t]{ \footnotesize minimum\\ \footnotesize maximum \\ \footnotesize references} &
\makecell[t]{ \scriptsize$S_n$($\lambda>0$), $P_n$($\lambda<0$) \\ \scriptsize $P_n$($\lambda>0$), $S_n$($\lambda<0$) \\ \footnotesize \cite{Nikoli}} &
\makecell[t]{  \textcolor[rgb]{0.00,0.07,1.00}{\scriptsize $CP_{n,k}^{\lceil\frac{k}{2}\rceil}$($\lambda>0$), $CP_{\lfloor\frac{n-k}{2}\rfloor,\lceil\frac{n-k}{2}\rceil}$($\lambda<0$)} \\  \scriptsize \textcolor[rgb]{0.00,0.07,1.00}{$CP_{\lfloor\frac{n-k}{2}\rfloor,\lceil\frac{n-k}{2}\rceil}$($\lambda>0$), $CP_{n,k}^{\lceil\frac{k}{2}\rceil}$($\lambda<0$)}  }&
\makecell[t]{ \scriptsize $CP_{n,k}^{\lceil\frac{k}{2}\rceil}$($\lambda>0$), \textcolor[rgb]{1.00,0.00,0.00}{?}($\lambda<0$)\\ \scriptsize $CP_{n,k}^{\lceil\frac{k}{2}\rceil}$($\lambda<0$), \textcolor[rgb]{1.00,0.00,0.00}{?}($\lambda>0$) \\ \footnotesize \cite{Zhang3}} &
\makecell[t]{ \scriptsize $SP_{n,q}$($\lambda>0$), $CP_{\lfloor\frac{q}{2}\rfloor,\lceil\frac{q}{2}\rceil}$($\lambda<0$)\\ \scriptsize  $CP_{\lfloor\frac{q}{2}\rfloor,\lceil\frac{q}{2}\rceil}$($\lambda>0$), $SP_{n,q}$($\lambda<0$) \\ \footnotesize  \cite{Zhang2}} &
\makecell[t]{ \scriptsize $CP_{n;1,\Delta-1}$($\lambda<0$), \textcolor[rgb]{1.00,0.00,0.00}{?}($\lambda>0$)\\ \scriptsize $CP_{n;1,\Delta-1}$($\lambda>0$), \textcolor[rgb]{1.00,0.00,0.00}{?}($\lambda<0$) \\ \footnotesize \cite{Liu1,Chen}}
\\
\midrule
\makecell[t]{ \footnotesize Variable \\ \footnotesize Wiener \\ \footnotesize indices } &
\makecell[t]{ \footnotesize minimum\\ \footnotesize maximum \\ \footnotesize references} &
\makecell[t]{ \scriptsize$S_n$($\lambda>1$), $P_n$($\lambda<1$) \\ \scriptsize $P_n$($\lambda>1$), $S_n$($\lambda<1$) \\ \footnotesize \cite{Vukic} } &
\makecell[t]{  \scriptsize \textcolor[rgb]{0.00,0.07,1.00}{$CP_{n,k}^{\lceil\frac{k}{2}\rceil}$($\lambda>1$), $CP_{\lfloor\frac{n-k}{2}\rfloor,\lceil\frac{n-k}{2}\rceil}$($\lambda<1$)} \\  \scriptsize \textcolor[rgb]{0.00,0.07,1.00}{$CP_{\lfloor\frac{n-k}{2}\rfloor,\lceil\frac{n-k}{2}\rceil}$($\lambda>1$), $CP_{n,k}^{\lceil\frac{k}{2}\rceil}$($\lambda<1$)}   } &
\makecell[t]{  \textcolor[rgb]{0.00,0.07,1.00}{\scriptsize $CP_{n,k}^{\lceil\frac{k}{2}\rceil}$($\lambda>1$)} \\  \scriptsize \textcolor[rgb]{0.00,0.07,1.00}{$CP_{n,k}^{\lceil\frac{k}{2}\rceil}$($\lambda<1$)} } &
\makecell[t]{  \scriptsize \textcolor[rgb]{0.00,0.07,1.00}{$SP_{n,q}$($\lambda>1$), $CP_{\lfloor\frac{q}{2}\rfloor,\lceil\frac{q}{2}\rceil}$($\lambda<1$)}\\  \scriptsize  \textcolor[rgb]{0.00,0.07,1.00}{$CP_{\lfloor\frac{q}{2}\rfloor,\lceil\frac{q}{2}\rceil}$($\lambda>1$), $SP_{n,q}$($\lambda<1$)} } &
\makecell[t]{  \scriptsize \textcolor[rgb]{0.00,0.07,1.00}{$CP_{n;1,\Delta-1}$($\lambda<1$)}\\  \scriptsize \textcolor[rgb]{0.00,0.07,1.00}{$CP_{n;1,\Delta-1}$($\lambda>1$)} }\\
\midrule
\makecell[t]{ \footnotesize  Steiner \\ \footnotesize  $k$-Wiener\\ \footnotesize  index } &
\makecell[t]{ \footnotesize minimum\\ \footnotesize maximum \\ \footnotesize references} &
\makecell[t]{ \footnotesize $S_n$\\ \footnotesize $P_n$ \\ \footnotesize \cite{Li}} &
\makecell[t]{  \footnotesize \textcolor[rgb]{0.00,0.07,1.00}{$CP_{n,k}^{\lceil\frac{k}{2}\rceil}$} \\  \footnotesize \textcolor[rgb]{0.00,0.07,1.00}{$CP_{\lfloor\frac{n-k}{2}\rfloor,\lceil\frac{n-k}{2}\rceil}$} } &
\makecell[t]{ \footnotesize $CP_{n,k}^{\lceil\frac{k}{2}\rceil}$ \\ \textcolor[rgb]{1.00,0.00,0.00}{?} \\ \footnotesize \cite{Lu}} &
\makecell[t]{ \footnotesize $SP_{n,q}$\\ \footnotesize $CP_{\lfloor\frac{q}{2}\rfloor,\lceil\frac{q}{2}\rceil}$ \\ \footnotesize \cite{Zhang0}}&
\makecell[t]{ \textcolor[rgb]{1.00,0.00,0.00}{?} \\ \footnotesize $CP_{n;1,\Delta-1}$ \\ \footnotesize\cite{Zhang0}}\\
\midrule
\makecell[t]{  \footnotesize hyper \\  \footnotesize -Wiener  \\  \footnotesize index } &
\makecell[t]{ \footnotesize minimum\\ \footnotesize maximum \\ \footnotesize references} &
\makecell[t]{  \footnotesize $S_n$\\ \footnotesize $P_n$ \\ \footnotesize \cite{Gutman2}} &
\makecell[t]{  \footnotesize \textcolor[rgb]{0.00,0.07,1.00}{$CP_{n,k}^{\lceil\frac{k}{2}\rceil}$} \\  \footnotesize \textcolor[rgb]{0.00,0.07,1.00}{$CP_{\lfloor\frac{n-k}{2}\rfloor,\lceil\frac{n-k}{2}\rceil}$} } &
\makecell[t]{ \footnotesize $CP_{n,k}^{\lceil\frac{k}{2}\rceil}$ \\ \textcolor[rgb]{1.00,0.00,0.00}{?} \\ \footnotesize \cite{Yu,Cai}} &
\makecell[t]{ \footnotesize $SP_{n,q}$\\ \footnotesize $CP_{\lfloor\frac{q}{2}\rfloor,\lceil\frac{q}{2}\rceil}$ \\ \footnotesize \cite{Yu}}&
\makecell[t]{ \textcolor[rgb]{1.00,0.00,0.00}{?} \\ \footnotesize $CP_{n;1,\Delta-1}$ \\ \footnotesize \cite{Yu}}\\
\midrule
\makecell[t]{ \footnotesize Wiener \\ \footnotesize -Hosoya \\ \footnotesize index } &
\makecell[t]{ \footnotesize minimum\\ \footnotesize maximum \\ \footnotesize references} &
\makecell[t]{ \footnotesize $S_n$\\ \footnotesize $P_n$ \\ \footnotesize \cite{Feng}} &
\makecell[t]{  \footnotesize \textcolor[rgb]{0.00,0.07,1.00}{$CP_{n,k}^{\lceil\frac{k}{2}\rceil}$} \\  \footnotesize \textcolor[rgb]{0.00,0.07,1.00}{$CP_{\lfloor\frac{n-k}{2}\rfloor,\lceil\frac{n-k}{2}\rceil}$} } &
\makecell[t]{ \footnotesize $CP_{n,k}^{\lceil\frac{k}{2}\rceil}$ \\ \textcolor[rgb]{1.00,0.00,0.00}{?} \\ \footnotesize \cite{Feng}} &
\makecell[t]{  \footnotesize \textcolor[rgb]{0.00,0.07,1.00}{$SP_{n,q}$}\\  \footnotesize \textcolor[rgb]{0.00,0.07,1.00}{$CP_{\lfloor\frac{q}{2}\rfloor,\lceil\frac{q}{2}\rceil}$}}&
\makecell[t]{ \textcolor[rgb]{1.00,0.00,0.00}{?} \\  \footnotesize \textcolor[rgb]{0.00,0.07,1.00}{$CP_{n;1,\Delta-1}$}}\\
\midrule
\makecell[t]{ \footnotesize degree \\ \footnotesize distance } &
\makecell[t]{ \footnotesize minimum\\ \footnotesize maximum \\ \footnotesize references} &
\makecell[t]{ \footnotesize $S_n$\\ \footnotesize $P_n$ \\ \footnotesize \cite{Tomescu}} &
\makecell[t]{  \footnotesize \textcolor[rgb]{0.00,0.07,1.00}{$CP_{n,k}^{\lceil\frac{k}{2}\rceil}$} \\  \footnotesize \textcolor[rgb]{0.00,0.07,1.00}{$CP_{\lfloor\frac{n-k}{2}\rfloor,\lceil\frac{n-k}{2}\rceil}$} } &
\makecell[t]{  \textcolor[rgb]{0.00,0.07,1.00}{\footnotesize $CP_{n,k}^{\lceil\frac{k}{2}\rceil}$} \\ \textcolor[rgb]{1.00,0.00,0.00}{?} } &
\makecell[t]{  \footnotesize \textcolor[rgb]{0.00,0.07,1.00}{$SP_{n,q}$}\\  \footnotesize \textcolor[rgb]{0.00,0.07,1.00}{$CP_{\lfloor\frac{q}{2}\rfloor,\lceil\frac{q}{2}\rceil}$}}&
\makecell[t]{ \textcolor[rgb]{1.00,0.00,0.00}{?} \\  \footnotesize \textcolor[rgb]{0.00,0.07,1.00}{$CP_{n;1,\Delta-1}$}}\\
\midrule
\makecell[t]{ \footnotesize Gutman\\ \footnotesize index } &
\makecell[t]{ \footnotesize minimum\\ \footnotesize maximum \\ \footnotesize references} &
\makecell[t]{ \footnotesize $S_n$\\ \footnotesize $P_n$ \\ \footnotesize\cite{Andova}} &
\makecell[t]{  \footnotesize \textcolor[rgb]{0.00,0.07,1.00}{$CP_{n,k}^{\lceil\frac{k}{2}\rceil}$} \\  \footnotesize \textcolor[rgb]{0.00,0.07,1.00}{$CP_{\lfloor\frac{n-k}{2}\rfloor,\lceil\frac{n-k}{2}\rceil}$} } &
\makecell[t]{  \footnotesize \textcolor[rgb]{0.00,0.07,1.00}{$CP_{n,k}^{\lceil\frac{k}{2}\rceil}$} \\ \textcolor[rgb]{1.00,0.00,0.00}{?} } &
\makecell[t]{ \footnotesize \textcolor[rgb]{0.00,0.07,1.00}{$SP_{n,q}$}\\  \footnotesize \textcolor[rgb]{0.00,0.07,1.00}{$CP_{\lfloor\frac{q}{2}\rfloor,\lceil\frac{q}{2}\rceil}$}}&
\makecell[t]{ \textcolor[rgb]{1.00,0.00,0.00}{?} \\  \footnotesize \textcolor[rgb]{0.00,0.07,1.00}{$CP_{n;1,\Delta-1}$}}\\
\midrule
\makecell[t]{ \footnotesize $ABC_2$\\ \footnotesize index } &
\makecell[t]{ \footnotesize minimum\\ \footnotesize maximum \\ \footnotesize references} &
\makecell[t]{ \footnotesize $P_n$\\ \footnotesize $S_n$ \\ \footnotesize\cite{Rostami}} &
\makecell[t]{  \footnotesize \textcolor[rgb]{0.00,0.07,1.00}{$CP_{\lfloor\frac{n-k}{2}\rfloor,\lceil\frac{n-k}{2}\rceil}$} \\ \footnotesize \textcolor[rgb]{0.00,0.07,1.00}{$CP_{n,k}^{\lceil\frac{k}{2}\rceil}$}} &
\makecell[t]{ \textcolor[rgb]{1.00,0.00,0.00}{?} \\  \footnotesize \textcolor[rgb]{0.00,0.07,1.00}{$CP_{n,k}^{\lceil\frac{k}{2}\rceil}$}} &
\makecell[t]{ \footnotesize \textcolor[rgb]{0.00,0.07,1.00}{$CP_{\lfloor\frac{q}{2}\rfloor,\lceil\frac{q}{2}\rceil}$}\\ \footnotesize \textcolor[rgb]{0.00,0.07,1.00}{$SP_{n,q}$}}&
\makecell[t]{ \footnotesize \textcolor[rgb]{0.00,0.07,1.00}{$CP_{n;1,\Delta-1}$} \\ \textcolor[rgb]{1.00,0.00,0.00}{?}}\\
\bottomrule
\end{tabularx}
\end{sidewaystable}

The extremal graphs given in Theorem \ref{thm-9-3} and Table 3 all  have specific structures whose topological indices can be simply calculated by the expression Eq. (\ref{WF-eq-22}). It is clear that the values of these topological indices of extremal graphs can be used as the sharp bounds for the corresponding classes of trees. As an example we give the bounds of Wiener index in the classes $\mathcal{C}(n,k)$, $\mathcal{T}_n(q)$ and $\mathcal{T}_n^\Delta$ in the following propositions, in which  the proof and calculation are omitted because they are conventional.

\begin{pro}\label{pro-1}
Let $T\in \mathcal{C}(n,k)$ where $k\geq2$, we have
$$W(T)\geq W(CP_{n,k}^{\lceil\frac{k}{2}\rceil})
=\left\{\begin{array}{ll}
-\frac{k^3}{12}+\frac{nk^2}{4}-kn+\frac{13k}{12}+n^2-\frac{5n}{4}, & \mbox{ if $k$ is odd}\\
-\frac{k^3}{12}+\frac{nk^2}{4}-kn+\frac{5k}{6}+n^2-n, &  \mbox{ if $k$ is even }
\end{array}\right.$$
and
$$W(T)\leq W(CP_{\lfloor\frac{n-k}{2}\rfloor,\lceil\frac{n-k}{2}\rceil})
=\left\{\begin{array}{ll}
-\frac{k^3}{12}+\frac{k^2}{4}+\frac{kn^2}{4}-kn+\frac{5k}{6}+\frac{3n^2}{4}-n, & \mbox{ if $n-k$ is even}\\
-\frac{k^3}{12}+\frac{k^2}{4}+\frac{kn^2}{4}-kn+\frac{7k}{12}+\frac{3n^2}{4}-n+\frac{1}{4},
&  \mbox{ if $n-k$ is odd.  }\\
\end{array}\right.$$
\end{pro}

\begin{table}[ht]
\footnotesize
\caption{\footnotesize $W(CP_{n,k}^{\lceil\frac{k}{2}\rceil})$ and $W(CP_{\lfloor\frac{n-k}{2}\rfloor,\lceil\frac{n-k}{2}\rceil})$ for $5\leq n\leq 11$ and $4\leq k\leq 10$}
\centering
\begin{tabular*}{15cm}{p{40pt}p{65pt}p{65pt}|p{50pt}p{65pt}p{65pt}}
\hline
 $(n,k)$ & $W(CP_{n,k}^{\lceil\frac{k}{2}\rceil})$ & $W(CP_{\lfloor\frac{n-k}{2}\rfloor,\lceil\frac{n-k}{2}\rceil})$ & $(n,k)$ & $W(CP_{n,k}^{\lceil\frac{k}{2}\rceil})$ & $W(CP_{\lfloor\frac{n-k}{2}\rfloor,\lceil\frac{n-k}{2}\rceil})$ \\\hline
$(5,4)$  &18  &$20$     & $(6,4)$  &28 &35    \\\hline
$(6,5)$  &31  &$35$     & $(7,4)$  &40 &52    \\\hline
$(7,5)$  &44  &$56$     & $(7,6)$  &50 &56    \\\hline
$(8,4)$  &54  &$74$     & $(8,5)$  &59 &79    \\\hline
$(8,6)$  &67  &$84$     & $(8,7)$  &75 &84    \\\hline
$(9,4)$  &70  &$98$     & $(9,5)$  &76 &108   \\\hline
$(9,6)$  &86  &$114$    & $(9,7)$  &96 &120   \\\hline
$(9,8)$  &108 &$120$    & $(10,4)$ &88 &127   \\\hline
$(10,5)$ &95  &$139$    & $(10,6)$ &107 &151   \\\hline
$(10,7)$ &119 &$158$    & $(10,8)$ &134 &165   \\\hline
$(10,9)$ &149 &$165$    & $(11,4)$ &108 &158   \\\hline
$(11,5)$ &116 &$176$    & $(11,6)$ &130 &190   \\\hline
$(11,7)$ &144 &$204$    & $(11,8)$ &162 &212   \\\hline
$(11,9)$ &180 &$220$    & $(11,10)$ &200 &220   \\\hline
\end{tabular*}
\end{table}

\begin{pro}
Let $T\in \mathcal{T}_n(q)$. Let $n-1=sq+r$ where $q\geq3$ and $0\leq r\leq q-1$, then
$$
W(T)\geq W(SP_{n,q})
=\frac{3q-2}{6}qs^3+(\frac{1}{2}q^2+\frac{3}{2}qr-r)s^2+(r^2+\frac{3}{2}qr+\frac{1}{3}q-r)s+r^2,
$$
and
$$W(T)\leq W(CP_{\lfloor\frac{q}{2}\rfloor,\lceil\frac{q}{2}\rceil})
=\left\{\begin{array}{ll}
\frac{n^3}{6}-\frac{nq^2}{4}+\frac{nq}{2}-\frac{n}{6}+\frac{q^3}{12}+\frac{q^2}{4}-\frac{5q}{6}, & \mbox{ if $q$ is even}\\
\frac{n^3}{6}-\frac{nq^2}{4}+\frac{nq}{2}-\frac{5n}{12}+\frac{q^3}{12}+\frac{q^2}{4}-\frac{7q}{12}+\frac{1}{4},
&  \mbox{ if $q$ is odd. }\\
\end{array}\right.$$
\end{pro}

\begin{pro}
Let $T\in \mathcal{T}_n^\Delta$, we have
\begin{equation}
\begin{aligned}
W(T)&\leq W(CP_{n;1,\Delta-1})\\ \nonumber
&=\left\{\begin{array}{ll}
\frac{\Delta^3}{3}-\frac{n+1}{2}\Delta^2+\frac{9n-5}{6}\Delta+\frac{n^3}{6}-\frac{7n}{6}+1, & \mbox{ if $\Delta<\lfloor\frac{n}{2}\rfloor$ and $\Delta>\lceil\frac{n}{2}\rceil$}\\
\frac{n^3}{12}+\frac{5n^2}{8}-\frac{19n}{12}+1,
&  \mbox{ if $n\geq4$ is even and $\Delta=\frac{n}{2}$ }\\
\frac{n^3}{12}+\frac{3n^2}{4}-\frac{25n}{12}+\frac{5}{4},
&  \mbox{ if $n\geq5$ is odd and $\Delta=\lfloor\frac{n}{2}\rfloor$ }\\
\frac{n^3}{12}+\frac{n^2}{2}-\frac{13n}{12}+\frac{1}{2},
&  \mbox{ if $n\geq5$ is odd and $\Delta=\lceil\frac{n}{2}\rceil$. }\\
\end{array}\right.
\end{aligned}
\end{equation}
\end{pro}

At last of this paper, in terms of the formulas in Proposition \ref{pro-1} we list the bounds  of the Wiener index  in the class $\mathcal{C}(n,k)$ for $5\leq n\leq 11$ and $4\leq k\leq 10$ in Table 4.

\end{document}